\documentclass[11pt,a4paper]{article}
\usepackage{authblk}

\usepackage{fullpage}

\usepackage[ruled]{algorithm2e} 

\SetAlFnt{\small}
\SetAlCapFnt{\small}
\SetAlCapNameFnt{\small}
\SetAlCapHSkip{0pt}
\IncMargin{-\parindent}

\usepackage{amsmath,amsfonts,amsthm}

\usepackage{tensor}
\usepackage{bm}
\usepackage{xcolor}

\usepackage{graphicx}
\usepackage{amsmath}
\usepackage{amssymb}
\usepackage{wrapfig}

\newcommand\ILSum[3]{
\underbrace{\sum\limits^{\lex}_{#2}\ldots\sum\limits^{\lex}_{#2}}_{#1 \text{ times}}{#3}
}

\newcommand\scc[1]{\mathrm{sc}(#1)}
\newcommand\Nmod[1]{{N(#1)}}
\def\lex{\mathrm{lex} }
\newcommand\LSuml[2]{{\sum\limits^{\lex}}_{#1}{#2}}
\newcommand\LSumIt[3]{{\sum\limits^{\lex^{#1}}}_{#2}{#3}}
\newcommand\LSumAll[2]{\LSuma{*}{#1}{#2}}
\newcommand\framets[1]{\mathsf{#1}}
\newcommand\modelsts[1]{\framets{#1}}
\newcommand\logicts[1]{{\textsc{#1}}}
\newcommand\languagets[1]{\logicts{#1}}
\newcommand\ClassNamets[1]{\mathrm{#1}}
\newcommand\LogicNamets[1]{\logicts{#1}}
\newcommand\reduct[2]{#1^{\restr #2}}
\newcommand\LSum[2]{\sum_{#1}{#2}}
\newcommand\LSuma[3]{\sum_{#2}\nolimits^{#1} #3}
\newcommand\myoper[1]{\mathop{\myopts{#1}}}
\newcommand\myopts[1]{\mathrm{#1}}
\newcommand\dom[1]{\myoper{dom}(#1)}
\newcommand\vct[1]{{\boldsymbol #1}}
\newcommand\Refin[2]{\myoper{Ref}(#1,#2)}
\newcommand\Disk[2]{{\bigsqcup}_{\leq #1}  #2}
\newcommand\Diskl[2]{{\bigsqcup}_{< #1}  #2}
\newcommand\clR[1]{{\myopts{sk}#1}}
\newcommand\ov[1]{\overline{#1}}
\newcommand\cone[2]{{#1}\!\left[#2\right]}
\newcommand\cmo[4]{ {#1, {#2\mo_{#3}\,}#4}  }
\newcommand\notcmo[4]{ #1, #2 \not\mo_{#3} #4 }
\newcommand\SF[1]{\myoper{sub}({#1})}
\newcommand\SFMG[3]{
{
{#1}[#2,#3]
}
}

\newcommand\trc[2]{ [#1 ]^{#2} }
\newcommand\trS[1]{ {#1}^\prime }
\newcommand\trDec[1]{ \dot{#1}}
\newcommand\univ[1]{#1^{(\AA)}}
\newcommand\univL[1]{#1^{ [\AA] }}
\newcommand\univLit[2]{#1^{(\AA)_{#2}}}
\newcommand\empL[1]{#1^{(\emp)}}
\newcommand\empLit[2]{#1^{(\emp)_{#2}}}

\newcommand\LK[1]{\LogicNamets{K#1}}
\newcommand\LS[1]{\LogicNamets{S#1}}
\def\wK4{\LogicNamets{wK4}}
\def\T0{\LogicNamets{wK4T}_0}
\def\DL{\LogicNamets{DL}}
\def\PSPACE{\mathrm{PSpace}}
\def\NP{\mathrm{NP}}
\def\PTime{\mathrm{PTime}}
\def\RedPSP{\leq_\mathrm{T}^\PSPACE}
\def\RedKarp{\leq_\mathrm{m}^\PTime}
\def\mM{\modelsts{M}}
\def\frF{\framets{F}}
\def\frS{\framets{S}}
\def\frG{\framets{G}}
\def\frH{\framets{H}}
\def\frI{\framets{I}}
\def\frJ{\framets{J}}
\def\frT{\framets{T}}

\def\clF{\mathcal{F}}
\def\clS{\mathcal{S}}
\def\clG{\mathcal{G}}
\def\clI{\mathcal{I}}
\def\clJ{\mathcal{J}}
\def\clC{\mathcal{C}}

\def\nin{\not \in}
\def\ML{\languagets{ML}}
\def\PV{\languagets{PV}}
\def\J{\LogicNamets{J}}
\def\J{{\normalfont\LogicNamets{J}}}
\def\GL{\LogicNamets{GL}}
\def\Grz{\LogicNamets{Grz}}
\def\GrzB{\LogicNamets{Grz.Bin}}
\def\GLP{\LogicNamets{GLP}}
\def\restr{{\upharpoonright}}
\def\vL{L}
\def\Frames{\myoper{Fr}}
\def\SubFrs{\myoper{Sub}}

\def\Log{\myoper{Log}}
\def\Di{\lozenge}
\def\Boxm{\Box^{\leq m}}
\def\imp{\rightarrow}
\def\vf{\varphi}
\def\mo{\models}
\def\se{\subseteq}
\def\toto{\twoheadrightarrow}
\def\isom{\cong}
\def\tiff{\text{ iff }}
\def\AA{\forall}
\def\EE{\exists}
\def\emp{\varnothing}
\def\empc{{(\emp)_\Al}}
\def\zervf{\mathbf{0}_\vf}
\def\oGamma{\vct{\Gamma}}
\def\oDelta{\vct{\Delta}}
\def\con{\wedge}

\def\h{{\mathop{ht}}}
\def\b{{\mathop{br}}}

\def\rfn{\rhd}
\def\lvf{{\#\vf}}
\def\NPO{\ClassNamets{NPO}}
\def\Trf{{\ClassNamets{Tr}_f}}

\def\Tr{\ClassNamets{Tr}}

\def\QOf{{\ClassNamets{QO}_f}}
\def\ltms{{\,\leftthreetimes\,}}
\def\fus{*}
\def\dts{\dots}
\def\Al{\mathrm{A}}
\def\AlA{\Al}
\def\AlB{\mathrm{B}}
\def\AlC{\mathrm{C}}
\def\AlM{\Al}

\newcommand\trSub[1]{[#1]_q}
\newcommand\trLad[1]{[#1]_\mathrm{L}}
\newcommand\trRel[1]{[#1]_{\myoper{Rel}}}

\def\ordforms{\sqsupseteq}

\def\vect{\mathbf{v}}
\def\vectu{\mathbf{u}}
\def\vectU{\mathbf{U}}
\newcommand\plusV[1]{\mathop{+^{#1}}}
\def\cupa{\mathop{\cup^a}}
\def\plusa{\plusV{a}}
\newcommand\Sat[1]{\myoper{Sat}#1}
\newcommand\CSat[1]{\myoper{CSat}#1}

\newcommand\fv[2]{#1 [#2 ] }
\newcommand\algts[1]{ \mathrm{#1} }
\def\true{\mathrm{true}}
\def\false{\mathrm{false}}

\def\BigO{O}

\def\CondSatSumFa{\algts{CSatSum}_\clF}
\def\CondSatF{\algts{CSatSummand}}
\def\CSatJ{\algts{CSatJ}}

\newtheorem{theorem}{Theorem}
\newtheorem{lemma}[theorem]{Lemma}
\newtheorem{proposition}[theorem]{Proposition}

\newtheorem{corollary}[theorem]{Corollary}

\theoremstyle{remark}
\newtheorem{example}{Example}

\newtheorem*{examp*}{Example}
\newtheorem*{remark*}{Remark}

\theoremstyle{definition}
\newtheorem{definition}{Definition}
\newtheorem*{definition*}{Definition}

\newcommand\hide[1]{\commented{gray}{Hidden:}{#1}}
\renewcommand\hide[1]\empty

\newcommand\todo[1]{\commented{blue}{Todo}{#1}}
\renewcommand\todo[1]\empty

\newcommand\comm[1]{\commented{magenta}{Note:}{{\it   #1}}}
\renewcommand\comm[1]\empty

\newcommand\improve[1]{\commented{blue}{Improve:}{#1}}
\renewcommand\improve[1]\empty

\begin{document}
\newtheorem{remark}[theorem]{Remark}

  \title{Satisfiability problems
  on sums of Kripke frames}

  \author{Ilya B.~Shapirovsky}
\affil{\small New Mexico State University, USA\\
Institute for Information Transmission Problems of Russian Academy of Sciences
}
\maketitle


\begin{abstract}
We consider the operation of sum on Kripke frames, where a family of frames-summands is  indexed by elements of another frame.
In many cases, the modal logic of sums inherits the finite model property and decidability from the modal logic of summands \cite{BabRyb2010}, \cite{AiML2018-sums}.
In this paper
we show that, under a general condition,
the satisfiability problem on sums is
polynomial space Turing reducible to the satisfiability problem on summands.
In particular, for many modal logics decidability in $\PSPACE$ is an immediate corollary from the semantic characterization of the logic.

\end{abstract}


\section{Introduction}
\improve{Say a few words about hardness - a good section}

In classical model theory, there is a number of results (``composition theorems'')
that reduce the theory (first-order, MSO) of a compound structure (e.g., sum or product) to the theories of its components, see, e.g.,
\cite{Mostowski1952,FefermanVaught1959,Shelah-GenSums75,Gurevich79ShortChainsI} or
\cite{GurevichChapter85}. In this paper we use the composition method in the context of modal logic.

Given a family ${(\frF_i \mid i  \text{ in }\frI)}$ of frames (structures with binary relations)
indexed by elements of another frame $\frI$,
the {\em sum  of the frames $\frF_i$'s
over $\frI$} is
obtained from their disjoint union by connecting elements of $i$-th and $j$-th distinct components according to the relations in $\frI$.
Given a class $\clF$ of frames-summands and a class $\clI$ of frames-indices, $\LSum{\clI}{\clF}$ denotes the class  of all sums of $\frF_i$'s in $\clF$ over
$\frI$ in $\clI$.

In many cases, this operation preserves the finite model property and decidability of the logic of summands \cite{BabRyb2010}, \cite{AiML2018-sums}.
In this paper
we show that transferring results also hold for the complexity of the modal satisfiability problems on sums.
In particular, it follows that for many logics $\PSPACE$-completeness is an immediate corollary of semantic characterization.

It is a classical result by R. Ladner that the decision problem for the logic of (finite) preorders $\LS{4}$ is in $\PSPACE$ \cite{Ladner77}.
In \cite{ShapJapar}, it was shown that the polymodal provability logic $\GLP$ is also decidable in $\PSPACE$.
In spite of the significant difference between these two logics,
there is a uniform proof for both these two particular systems, as well as for many other important modal logics:
the general phenomenon is that
the modal satisfiability
problem on sums
over Noetherian (in particular, finite) orders
is polynomial space Turing reducible to the
modal satisfiability
problem on summands.
In the case of $\LS{4}$ these summands are frames of the form $(W,W\times W)$ with the satisfiability problem being in $\NP$, so in $\PSPACE$. And hence,  $\LS{4}$ is in $\PSPACE$.
In the case of $\GLP$, the class of summands is even simpler: the only summand required is an irreflexive singleton \cite{Bekl-Jap}.
(We will discuss these and other examples in Sections \ref{sub:examples}, \ref{subs:Japar}, \ref{sub:refandlexProduct}.)

The paper has the following structure. Section \ref{sec:prel} contains preliminary material.
Section \ref{sec:sums}  is about truth-preserving operations on sums of frames; it contains
necessary tools for the complexity results, and also quotes recent results on the finite model property. This section is based on \cite{AiML2018-sums}.
The central Section 4 is about complexity.
The reduction between sums and summands is described in
Theorem \ref{thm:main_compl_res} and in its more technical (but more tunable) version Theorem \ref{Thm:CorrectAlgTree}; these are the main results of the paper.
This section elaborates earlier works \cite{ShapJapar,ShapPSPACE05} (in particular,
our new results
significantly generalize and simplify Theorems 22 and 35 in \cite{ShapJapar}).
In Section \ref{sec:variations} we discuss
modifications of the sum operation: (iterated) lexicographic sums of frames, which are important in the context of provability logics \cite{Bekl-Jap};
lexicographic products of frames, earlier studied in \cite{Balb2009,BalbLinAx2010,BalbianiMik2013,BalbDuque2016}; the operation of refinement of modal logics introduced in \cite{BabRyb2010}.
Further results and directions are discussed in Section \ref{sec:concl}.

\improve{
In .. we apply it to obtain a uniform proof for ... (these results are well-known) and for logics ....
In Section ... we discuss the lower bound; this part is simple, since all the required construction has already been provided in
Lad and Spaan \cite{...}, and we only need to extract a convenient formulation for sums.
}

\hide{

This natural operation (being a particular case of generalized sum of models introduced by S. Shelah in the 1970th) has had several important applications in the context of modal logic over the last decade. Iterated sums of Noetherian orders were used by L. Beklemishev to study models of the polymodal provability logic. Then it was noted by the author that sums of frames provide a natural way to study computational complexity of modal satisfiability problems. S.~Babenyshev and V.~ Rybakov considered an operation on frames called {\em refinement}, and showed that under a very general condition this operation preserves the finite model property and decidability of logics; refinements can be considered as special instances of sums of frames. The lexicographic product of modal logics, introduced by Ph. Balbiani, is another example of an operation that can be defined via sums of frames.

In this talk, I, first, will discuss general semantic tools for operating with sums of Kripke frames,
and apply them to study the finite model property and decidability.
Then, I will provide conditions under which the modal satisfiability problem on
$\LSum{\clI}{\clF}$ is polynomial space Turing reducible to the modal satisfiability problem on $\clF$.

\comm{Shelah, products etc}

}

\section{Preparatory syntactical and semantical definitions}\label{sec:prel}
\improve{This title if due to Mostowski}

Let $\AlA\leq \omega$. The set $\ML(\AlA)$ of {\em modal formulas over $\AlA$} (or {\em $\AlA$-formulas}, for short)  is built from
a countable set of {\em variables} $\PV=\{p_0,p_1,\ldots\}$ using Boolean connectives $\bot,\imp$ and unary connectives $\Di_a$, $a< \AlA$ ({\em modalities}).
The connectives
$\vee,\wedge, \neg,\top,\Box_a$
are defined as abbreviations in the standard way, in particular $\Box_a\vf$ is $\neg\Di_a\neg\vf$.
\improve{When $A$ is a singleton, we omit the index for modalities. \improve{Do we need it?}}

\medskip

An ($\AlA$-)frame is a structure  $\frF=(W,(R_a)_{a< \AlA})$,
where ${W\ne\varnothing}$ and ${R_a\se W{\times}W}$ for $a< \AlA$. A \emph{(Kripke) model on} $\frF$ is a pair $\mM=(\frF,\theta)$, where $\theta:\PV\to 2^W$.
We write $\dom{\frF}$ for $W$, which is called the {\em domain} of $\frF$.
We write $u\in \frF$ for $u\in \dom{\frF}$.
For $u\in W$, $V\subseteq W$, we put $R_a(u)=\{v\mid uR_av\}$,
$R_a[V]=\cup_{v\in V} R_a(v)$.

The \emph{truth relation} in a model is defined in the usual way, in particular
${\mM,w\mo\Di_a\vf}$ means that  ${\mM,v\mo\vf}$  for some $v$ in $R_a(w)$.
A formula $\vf$ is {\em satisfiable in a model $\mM$} if $\mM,w\mo\vf$ for some $w$ in $\mM$.
A formula is {\em satisfiable in a frame $\frF$} ({\em in a class $\clF$ of frames})
if it is satisfiable in some model on $\frF$ (in some model on a frame in $\clF$).
$\vf$ is valid in a frame $\frF$ (in a class $\clF$ of frames) if $\neg \vf$ is not satisfiable in $\frF$ (in $\clF$).
Validity of a set of formulas means validity of every formula in this set.

The notions of p-morphism, generated subframe and submodel are defined in the standard way, see e.g. \cite[Section 1.4]{ManyDim}.
The notation $\frF\toto \frG$ means that $\frG$ is a p-morphic image of $\frF$;  $\frF\isom \frG$ means that
$\frF$ and $\frG$ are isomorphic.
\hide{We write $\cone{\frF}{w}$ for the subframe of $\frF$ generated by the singleton $\{w\}$;
such frames are called {\em cones in $\frF$}.
}

\medskip

A ({\em propositional normal modal}) {\em logic} is a set $\vL$ of formulas
that contains all classical tautologies, the axioms $\neg\Di_a \bot$  and
$\Di_a (p_0\vee p_1) \imp \Di_a p_0\vee \Di_a p_1$ for each $a$  in $\AlA$, and is closed under the rules of modus ponens,
substitution and monotonicity (if $\vf\imp\psi\in \vL$, then $\Di_a\vf\imp\Di_a\psi\in \vL$, for each $a$  in $\AlA$).

The set $\Log{\clF}$ of all formulas that are valid in a class $\clF$ of $\AlA$-frames is a logic (see, e.g., \cite{CZ}); it is called the {\em logic of $\clF$};
such logics are called {\em Kripke complete}.
A logic has the {\em finite model property}
if it is the logic of a class of finite frames (a frame is finite, if its domain is).
$\Frames{\vL}$ denotes the class of all frames validating $\vL$.

Remark that for a Kripke complete logic, its decision problem is the
validity problem on the class of all its frames (or on any other class $\clF$ such that $L=\Log(\clF)$).
The dual problem is the satisfiability:
$\Sat(\clF)$ is the set of all formulas satisfiable in $\clF$ (in the signature of $\clF$).
Remark that for the class $\PSPACE$ (as well as for any other deterministic complexity class),
$\Sat{\clF}\in \PSPACE$ iff $\Log \clF\in \PSPACE$.

\medskip
Natural numbers are considered as finite ordinals.  Given a sequence $\vct{v}=(v_0,v_1,\ldots)$, we write $\vct{v}(i)$ for $v_i$.

 \todo{
$\LS{5}.$
$\LS{4.1}$, $\Grz.2$
}

\section{Sums}\label{sec:sums}

We fix  $\Al\leq \omega$\improve{Why $\leq \omega?$} for the alphabet and consider the language $\ML(\Al)$.

\comm{old; in particular, old notations ($N$ for $\Al$)}
Consider a non-empty family $(\frF_i)_{i\in I}$ of $\Al$-frames
$\frF_i=(W_i,(R_{i,a})_{a\in \Al})$.
The {\em disjoint union} of these frames is the $\Al$-frame
$\bigsqcup_{i\in I} \frF_i=(\bigsqcup_{i\in I}{W_i},(R_a)_{a\in \Al})$, where $\bigsqcup_{i\in I}{W_i}=\bigcup_{i\in I}(\{i\}\times W_i)$
is the  disjoint union of sets $W_i$,
and $$(i,w)R_a (j,v) \quad \tiff \quad  i = j \,\&\,w R_{i,a} v.$$

Suppose that $I$  is the domain of another $\Al$-frame $\frI=(I,(S_a)_{a\in \Al})$.
\begin{definition}\label{def:sum-poly-extra}
The {\em sum of the family $(\frF_i)_{i\in I}$ of $\Al$-frames over the $\Al$-frame $\frI=(I,(S_a)_{a\in \Al})$} is the $\Al$-frame
 $\LSum{i\in \frI}{\frF_i}=(\bigsqcup_{i \in I} W_i, (R^\Sigma_a)_{a\in \Al})$, where
$$(i,w)R^\Sigma_a (j,v) \quad \tiff \quad (i = j \,\&\,w R_{i,a} v) \text{ or } (i\neq j \,\&\, iS_a j).$$
{\em  The sum of models} $\LSum{i\in\frI}{(\frF_i,\theta_i)}$ is the model $(\LSum{i\in\frI}{\frF_i},\theta)$,
where $(i,w)\in \theta(p)$ iff $w\in \theta_i(p)$.

For  classes $\clI$, $\clF$  of $\Al$-frames,
let $\LSum{\clI}{\clF}$  be the class of all sums
$\LSum{i\in\frI}{\frF_i}$ such that $\frI \in \clI$ and  $\frF_i\in \clF$ for every $i$ in $\frI$.
\end{definition}

\improve{

\begin{wrapfigure}{R}{0.4\textwidth}
\begin{center}
\includegraphics[width=0.45\textwidth]{sum-unimod-c.jpg}
\end{center}
\caption*{A unimodal sum}
\end{wrapfigure}
}

\begin{remark}\label{rem:independOfRefl}
We do not require that $S_a$'s are partial orders or even transitive relations.
Also, we note that the relations $R^\Sigma_a$ are independent of reflexivity of the relations $S_a$:
if  $\frI'=(I,(S'_a)_{a\in N})$, where $S'_a$ is the reflexive closure of $S_a$ for each $a\in \Al$,
then $\LSum{i\in \frI}{\frF_i}=\LSum{i\in \frI'}{\frF_i}$.
\end{remark}

\begin{example}[Skeleton and clusters]
A {\em cluster} is a frame of the form $(W,W\times W)$.

Let $\frF=(W,R)$ be a preorder.
The {\em skeleton} of $\frF$ is the partial order $\clR{\frF}=(\ov{W},\leq_R)$, where
$\ov{W}$ is the quotient set of $W$ by the equivalence $R\cap R^{-1}$, and for
$C,D\in \ov{W}$, $C\leq_R D$ iff $\EE w\in C \,\EE v\in D\; wRv$.
The restriction of $\frF$ on an element
of $\ov{W}$ is  called a {\em cluster in $\frF$}.

It is easy to see that every preorder $\frF$ is isomorphic to the sum $\LSum{C\in \clR{\frF}}(C,C\times C)$
of its clusters over its skeleton. \improve{Formally. clusters are not frames. Perhaps, include the definition of skeleton in Prels}
\end{example}

\begin{example}\label{ex:wk4semant}
Suppose that $\frF=(W,R)$ satisfies the property of {\em weak transitivity }
\mbox{$x R z R  y\, \Rightarrow \,xR y\vee x=y.$}
Let $\frI$ be the skeleton of the preorder $(W,R^*)$, where $R^*$
is the transitive reflexive closure of $R$.
Then $\frF$ is isomorphic to a sum $\LSum{i\in\frI}{\frF_i}$ such that every $\frF_i=(W_i,R_i)$ satisfies the property
$x\neq y\,\Rightarrow\, xR_i y $.
\end{example}

We shall be mainly interested in the polymodal case where
the indexing frame has only one non-empty relation.
\begin{definition}\label{def:asum}
Consider a unimodal frame $\frI=(I,S)$ and a family $(\frF_i)_{i\in I}$ of $\Al$-frames (or $\Al$-models).
For $a\in \AlA$,
the {\em $a$-sum} $\LSuma{a}{\frI}{\frF_i}$ is the sum
$\LSum{\frI'}{\frF_i}$, where
$\frI'$ is the $\Al$-frame whose domain is $I$,
the $a$-th relation is $S$ and other relations are empty.
If $\clF$ is a class of $\Al$-frames,  $\clI$ is a class of 1-frames, then $\LSuma{a}{\clI}{\clF}$
is the class of all sums
$\LSuma{a}{\frI}{\frF_i}$, where $\frI \in \clI$ and all $\frF_i$ are in $\clF$.\improve{Unravel}
\end{definition}

\improve{S42?}

\improve{
As we shall prove in Theorem \ref{THMAIN}

As we shall prove in Theorem \ref{THMAIN}, the modal satisfiability problem on preorders (i.e., the satisfiability problem for the logic $\LS{4}$) can be reduced (....) to the satisfiability
in clusters.  \comm{Now - $\LS{5}$; see slides in LC2019}
The situation ..... with $\wK4$ is ....:

; for the logic ..., the satisfiability is reducible to .... In both these  ,... In general,
}

\subsection{Basic truth-preserving operations}

The theorem below is a collection of facts illustrating
how sums interact with p-morphisms, generated subframes, and disjoint unions. Their proofs are straightforward from definitions,
see \cite[Section 3]{AiML2018-sums} for the details.
\begin{theorem}[\cite{AiML2018-sums}]~\label{thm:basicTrPr}
\begin{enumerate}
    \item  Let $\frI$ be an $\Al$-frame and
    $(\frF_i)_{i\in\frI}$ a family of $\Al$-frames.
    If $\frJ$  is a generated subframe of $\frI$, then  $\LSum{i\in \frJ}{\frF_i}$ is a generated subframe of $\LSum{i\in \frI}{\frF_i}$.

    \item
Let  $\frI$, $\frJ$ be $\Al$-frames,
$(\frF_i)_{i\in \frI}$, $(\frG_j)_{j\in \frJ}$
families of $\Al$-frames.
Suppose that all the relations in $\frJ$ are irreflexive.
\begin{enumerate}
\item
If
$f: \frI\toto \frJ$ and
$\frF_i\toto\frG_{f(i)}$ for all $i$ in $\frI$, then $\LSum{i\in \frI}{\frF_i}\toto\LSum{j\in \frJ}{\frG_j}.$
\item \label{prop:psp:pmorph-sum-extra-2}
If $\frI=\frJ$ and $\frF_i\toto\frG_i$ for all $i$ in $\frI$, then $\LSum{i\in\frI}{\frF_i}\toto\LSum{i\in\frI}{\frG_i}.$
\item \label{prop:psp:pmorph-sum-extra-3}
If $f:\frI\toto \frJ$, then $\LSum{i\in \frI}{\frG_{f(i)}}\toto\LSum{j\in \frJ}{\frG_j}.$
\end{enumerate}
\item\label{item3:basicTrPr}
\begin{enumerate}
    \item Let $\frI$ be an $\Al$-frame,  $(\frJ_i)_{i\in \frI}$ a family of $\Al$-frames, and
$(\frF_{ij})_{i\in \frI, j\in \frJ_i}$
a family
of $\Al$-frames. Then
$$\LSum{i \in \frI}{\LSum{j\in \frJ_i} {\frF_{ij}}}
\quad \isom\quad
\LSum{(i,j)\in \LSum{k \in \frI}{\frJ_k}}{\frF_{ij}}.$$

\item
Let   $I$ be a non-empty set, $(\frJ_i)_{i\in \frI}$ a family of $\Al$-frames, and
$(\frF_{ij})_{i\in \frI, j\in \frJ_i}$
a family
of {$\Al$-frames}. Then
$${\bigsqcup}_{i \in I}{\LSum{j \in \frJ_i} {\frF_{ij}}}
\quad\isom\quad
\LSum{(i,j)\in {\bigsqcup}_{k \in I}{\frJ_k}}{\frF_{ij}}.$$

\item\label{item3c:basicTrPr}
Let $\frI$ be an $\Al$-frame, $(J_i)_{i\in \frI}$ a family  of non-empty sets, and
$(\frF_{ij})_{i\in \frI, j\in J_i}$ a family  of $\Al$-frames. Then
$$\LSum{i \in  \frI}{{\bigsqcup}_{j \in J_i} \frF_{ij}}
\quad\isom\quad
\LSum{(i,j)\in \LSum{k \in \frI}{(J_k,\empc)}}{\frF_{ij}},$$
where $\empc$ denotes the sequence of length $\Al$ in which every element is the empty set.
\end{enumerate}
\end{enumerate}

\end{theorem}

\subsection{Sums and universal modality}
\improve{and the finite model property}
\hide{
Classes of $\Al$-frames $\clF$ and $\clG$ are said to be {\em interchangeable}, in symbols $\clF\equiv\clG$,
if $\clF$ and $\clG$ have the same modal logic in the
language enriched with the universal modality.  Formally,}
For an $\Al$-frame $\frF=(W,(R_0,R_1,\ldots))$,  let $\univ{\frF}$ be the
$(1+\Al)$-frame $(W,(W\times W, R_0,R_1,\ldots))$.
For a class  $\clF$  of $\Al$-frames, $\univ{\clF}=\{\univ{\frF}\mid \frF\in\clF\}$.
\hide{We put $$\clF\equiv\clG  \text{ if  }\Log{\univ{\clF}}=\Log{\univ{\clG}}.$$
}

In \cite{AiML2018-sums}, it was shown that if the classes
$\univ{\clF}$ and $\univ{\clG}$ have the same logic, then for any class $\clI$ of $\Al$-frames,
the logics of sums $\LSum{\clI}{\clF}$ and $\LSum{\clI}{\clG}$ are equal; moreover,
the logics of the classes $\univ{\LSum{\clI}{\clF}}$ and $\univ{\LSum{\clI}{\clG}}$ are
equal, thus we have $\Log\LSum{\clJ}{(\LSum{\clI}{\clF})}
=\Log\LSum{\clJ}{(\LSum{\clI}{\clG})}$
for any other class of frames-indices $\clJ$, and so on.

\begin{theorem}[\cite{AiML2018-sums}, Theorem 4.11]\label{thm:interch1}
Let $\clI$, $\clF$, $\clG$ be classes of $\Al$-frames.
If $\Log{\univ{\clF}}=\Log{\univ{\clG}}$, then
$\Log\univ{\LSum{\clI}{\clF}}=\Log\univ{\LSum{\clI}{\clG}}$, and hence
$\Log\LSum{\clI}{\clF}=\Log\LSum{\clI}{\clG}$.
\end{theorem}

In particular, it follows that if the logic of the class $\univ{\clF}$ has the finite model property,
then the logic of the class of sums $\LSum{\clI}{\clF}$ is equal to the logic of the class of sums
$\LSum{\clI}{\clG}$, where $\clG$ is a class of finite frames.

\improve{
\todo{Give illustration for the FMP}

\todo{Collect all general facts which can be formulated without conditional satisfiability.

- No, move to FMP subsection.

It seems that this is a good place to introduce Japaridze's logic, the logic J, and then introduce
the notation for a-sums. Then to formulate the fmp results.
}
}
\subsection{Decomposition of sums}


\comm{Def from AiML}

To reduce satisfiability in sums to the satisfiability in summands, we will
use an auxiliary notion: {\em satisfiability under conditions}.

\begin{definition}\label{def:sat-mod}
A sequence $\oGamma=(\Gamma_a)_{a\in \AlM}$, where $\Gamma_a$ are sets of $\AlM$-formulas,
is called a {\em condition} (in the language $\ML(\AlM)$).

Consider a model
$\mM=(W,(R_a)_{a\in \AlM},\theta)$,
 $w$  in $\mM$.
By induction on
the length
of an $\AlM$-formula $\vf$, we define the relation $\cmo{\mM}{w}{\oGamma}{\vf}$
(``{\em under the condition $\oGamma$, $\vf$ is true at $w$ in $\mM$''}): as usual,
$\notcmo{\mM}{w}{\oGamma}{\bot}$, $\cmo{\mM}{w}{\oGamma}{p}  \tiff  \mM,w\mo p$ for a variable $p$,
$\cmo{\mM}{w}{\oGamma}{\vf\imp \psi} \tiff  \notcmo{\mM}{w}{\oGamma}{\vf} \text{ or }
\cmo{\mM}{w}{\oGamma}{\psi}$; for $a\in \AlM$,
$$\cmo{\mM}{w}{\oGamma}{\Di_a \vf} \quad \tiff \quad \vf \in\Gamma_a \text{ or }
\EE v\in R_a(w)~~\cmo{\mM}{v}{\oGamma}{\vf}.$$
\end{definition}

In particular, if all $\Gamma_a$ are empty, then we have the standard notion of truth in a Kripke model:
$$\cmo{\mM}{w}{\empc}{\vf} \quad \tiff  \quad \mM,w\mo \vf,$$
where $\empc$ denotes the condition consisting of empty sets.
\comm{sequence of length $\AlM$ in which every element is the empty set. - it is unclear what is length $\AlM$, if $\alpha$ is not an ordinal}
The truth under conditions is respected by the standard operations on Kripke models:
\comm{Preamble}
\begin{proposition}\label{prop:truth-pres-for-cs}
Let $\oGamma$ be a condition and $\vf$ a formula.
\begin{enumerate}
\item
  If $\mM'$ is a generated submodel of $\mM$, then
$\cmo{\mM'}{w}{\oGamma}{\vf}$ iff $\cmo{\mM}{w}{\oGamma}{\vf}$ for every $w$ in $\mM'$.
\item

If $\mM=\bigsqcup_{i\in I}\mM_i$, then
$\cmo{\mM}{(i,w)}{\oGamma}{\vf}$ iff $\cmo{\mM_i}{w}{\oGamma}{\vf}$ for every $i$ in $I$ and every $w$ in $\mM_i$.
\item
If $f:\mM\toto \mM'$, then
$\cmo{\mM}{w}{\oGamma}{\vf}$ iff $\cmo{\mM'}{f(w)}{\oGamma}{\vf}$
for every $w$ in $\mM$.
\end{enumerate}
\end{proposition}
\begin{proof}
This proof completely reflects the proof of these facts for the standard truth relation in Kripke models and can
be obtained by a straightforward induction on the length of $\vf$.
\end{proof}

Let $\SF{\vf}$ be the set of all subformulas of $\vf$. We put
$$\SFMG{\vf}{\mM}{\oGamma}={\{\psi\in \SF{\vf}\mid\cmo{\mM}{v}{\oGamma}{\psi} \text{ for some } v\}}.$$
\improve{this set of subformulas of $\vf$ is called the {\em characterization of $\vf$ with respect to $\mM$ and $\oGamma$}.}
In particular,  $\SFMG{\vf}{\mM}{\empc}$
 is the set of all subformulas of $\vf$ that are satisfiable in $\mM$.

A triple $(\vf,\Phi,\oGamma)$, where $\Phi\subseteq \SF{\vf}$,   is called a {\em tie}.
A tie $(\vf,\Phi,\oGamma)$ is {\em satisfiable} in a frame $\frF$ (in a class $\clF$ of frames)
if there exists a model $\mM$ on $\frF$ (on a frame in $\clF$)  such that
$\Phi=\SFMG{\vf}{\mM}{\oGamma}$.
Hence, we have:

\begin{proposition}\label{prop:sat_via_cSat}
A formula $\vf$ is satisfiable in a class $\clF$ of frames iff
there exists $\Phi\subseteq   \SF{\vf}$ such that
$\vf\in \Phi$ and the tie $(\vf,\Phi,\empc)$ is satisfiable in $\clF$.
\end{proposition}
Classes of $\Al$-frames $\clF$ and $\clG$ are said to be {\em interchangeable}, in symbols $\clF\equiv\clG$, if the same ties are satisfiable in classes
$\clF$ and $\clG$.
From the above proposition, it follows that if $\clF\equiv \clG$, then the logics of these classes are equal.
Moreover, the following holds:
\begin{proposition}[\cite{AiML2018-sums}, Proposition 4.10]
Let $\clF$ and $\clG$ be classes of $\Al$-frames.
The following are equivalent:
\begin{itemize}
    \item  A tie is satisfiable in $\clF$ iff it is satisfiable in $\clG$.
    \item A formula is satisfiable in $\univ{\clF}$ iff it is
    satisfiable in $\univ{\clG}$.
\end{itemize}
\end{proposition}
The latter condition means that the logics of $\univ{\clF}$  and $\univ{\clG}$ are equal.
The term `interchangeable' is motivated by Theorem \ref{thm:interch1}.

Hence, the language of ties is as expressible as the modal language with the universal modality.
In Section \ref{Sec:SatviaCSat}, we will present explicit reductions between these languages.

\improve{
\todo{announce the following construction properly - important for the fmp, decidability, etc (see Theorem above).
And will be used in the proof of Theorem main }
}

\begin{definition}
Let $V$ be a set of elements of a model $\mM=(W,(R_a)_{a\in \AlM},\theta)$.
Given a formula $\vf$ and a condition $\oGamma$,
let $\oDelta$ be the condition defined as follows:   for $a\in \AlM$,
$$\oDelta(a)=\oGamma(a)\cup \{\chi\in \SF{\vf}\mid \EE w\in R_a[V]{\setminus}V\;  \cmo{\mM}{w}{\oGamma}{\chi}\}.$$
$\oDelta$ is called the {\em external condition of $V$ in $\mM$ with respect to $\vf$ and $\oGamma$}.
\end{definition}
\begin{remark}
If $\oGamma \subseteq \SF{\vf}^\Al$, then $\oDelta \subseteq \SF{\vf}^\Al$ as well (this will be important in the next
section, where we consider the {\em conditional satisfiability problem}).
\end{remark}
\improve{\todo{Reformulate for sums via Indexes}}
\begin{lemma}[\cite{AiML2018-sums}, Lemma 4.5] \label{psp:lem:extrernal}
Consider a sum of models $\mM=\LSum{\frI}{\mM_i}$, $i$ in $\frI$, and the set $V=\{i\}\times \dom{\mM_i}$.
\improve{Wording: Let ...}
If
$\oDelta$ is the external condition of $V$ in $\mM$ with respect to some given $\vf$, $\oGamma$, then
for all $v$ in $\mM_i$, $\chi$ in $\SF{\vf}$,
\begin{equation}\label{eq:satsumm0-poly-extra-1}
\cmo{\mM}{(i,v)}{\oGamma}{\chi} \quad \tiff \quad \cmo{\mM_i}{v}{\oDelta}{\chi}.
\end{equation}
\end{lemma}
\improve{The proof of Theorem \ref{thm:interch1} is entirely based on this lemma,
see \cite{AiML2018-sums}. -- who cares?
}

Lemma \ref{lem:satsumm-decomp-poly} below is
a particular corollary of Lemma \ref{psp:lem:extrernal}. It will be important for the proofs of our complexity results.

Consider $a\in \AlM$ and models $\mM_0$, $\mM_1$.
The model $\mM_0\plusa\mM_1$
 is obtained from the disjoint union of $\mM_0$ and $\mM_1$ by adding all the pairs of form
$((0,w),(1,v))$ to the $a$-th relation; that is,  in our general notation,
$\mM_0\plusa\mM_1=\LSuma{a}{(2,<)}{\mM_i}$.

For an $\AlM$-condition $\oGamma=(\Gamma_0,\ldots,\Gamma_{\AlM-1})$ and a set of $\AlM$-formulas $\Psi$, we put $\oGamma \cupa \Psi =
(\Gamma'_0,\ldots,\Gamma'_{\AlM-1})$, where $\Gamma'_a=\Gamma_a\cup \Psi$,
and $\Gamma'_b=\Gamma_b$ for $b\neq a$.

\comm{This lemma is great. It should be moved (to a special subsection), announced and explained properly}
\begin{lemma}\label{lem:satsumm-decomp-poly}
For $\AlM$-models $\mM_0, \mM_1$, an $\AlM$-formula $\vf$, an $\AlM$-condition $\oGamma$, and $a\in \AlM$, we
have
\begin{equation}\label{eq:sum-cup-sum}
\SFMG{\vf}{\mM_0\plusa\mM_1}{\oGamma}=\SFMG{\vf}{\mM_0}{\oGamma\cupa\SFMG{\vf}{\mM_1}{\oGamma}}\cup \SFMG{\vf}{\mM_1}{\oGamma}.
\end{equation}
\end{lemma}
\begin{proof}
Let $V$ be the bottom part of the sum $\mM_0\plusa\mM_1$: $V=\{(0,w)\mid w \text{ is in } \mM_0\}$.
Then $$\oDelta = \oGamma\cupa\SFMG{\vf}{\mM_1}{\oGamma}$$ is the
external condition of $V$ w.r.t.  $\vf$ and $\oGamma$.
The external condition of the top part $\{(1,v)\mid v \text{ is in } \mM_1\}$ w.r.t.  $\vf$ and $\oGamma$ is
just $\oGamma$.
By Lemma \ref{psp:lem:extrernal}, we have
for every $w$ in $\mM_0$,  every $v$ in $\mM_1$, and every $\chi\in \SF{\vf}$:
\begin{eqnarray*}
\cmo{\mM_0\plusa\mM_1}{(0,w)}{\oGamma}{\chi} &\quad \tiff\quad &\cmo{\mM_0}{w}{\oDelta}{\chi}\\
\cmo{\mM_0\plusa\mM_1}{(1,v)}{\oGamma}{\chi} &\quad \tiff \quad &\cmo{\mM_1}{v}{\oGamma}{\chi}.
\end{eqnarray*}
Now (\ref{eq:sum-cup-sum}) follows.
\end{proof}

\subsection{Sums over Noetherian orders}
We say that $\frI$  is a {\em Noetherian}  (or {\em converse well-founded}) {\em order} if
$\frI$  is a strict partial order which  has no infinite ascending chains. 
Let $\NPO$ be the class of all non-empty Noetherian orders  (we say that a partial order is non-empty, if its domain is).

A strict partial order $(I,<)$ is called a ({\em transitive irreflexive}) {\em tree}
if it has a least element (the {\em root}) and
for all $i\in I$ the set $\{j\mid j<i\}$ is a finite chain.
Let $\Trf$  be the class of all finite
trees.

Consider a finite tree $\frI=(I,<)$.
The {\em branching of $i$ in $\frI$}, denoted by $\b(i,\frI)$,
is the number of immediate successors of $i$ ({\em $j$ is an immediate successor of $i$}, if $i<j$ and there is no $k$ such that $i<k<j$);
the {\em branching of $\frI$}, denoted by $\b(\frI)$, is $\max{\{\b(i,\frI)\mid i \text{ in }\frI\}}$.
The {\em height of $\frI$}, denoted by $\h(\frI)$, is $\max{\{|V|\mid V \text{ is a chain in }\frI\}}$.
For $h,b\in \omega$, let $\Tr(h,b)$ be the class of all finite trees with height $\leq h$ and branching $\leq b$:
$${\Tr(h,b)=\{\frI\in \Trf \mid \h(\frI)\leq h \,\&\,\b(\frI)\leq b\}.}$$

Let $\bigsqcup \clF$ be the class of frames of form $\bigsqcup_I \frF_i$, where
$I$ is a non-empty set,
$\frF_i\in \clF$ for all $i\in I$, and let
$\Disk{k}{\clF}$ be the class of such frames with $0<|I|\leq k$. Likewise for $\Diskl{k}{\clF}$.

Let $\lvf$ denote the number of subformulas of a formula $\vf$.

\begin{theorem}[\cite{AiML2018-sums}]\label{thm:smll-tree-for-Csat}
Let
$\clF$ be a class of $\AlM$-frames, $a\in\AlM$,
  and $\clI$  a class of Noetherian orders containing all finite trees.
\begin{enumerate}
\item \label{thm:smll-tree-for-Csat-itemSat}
We have
$$\Log \LSuma{a}{\NPO}{\clF} =\Log\LSuma{a}{\clI}{\clF}=\Log\LSuma{a}{\Trf}{\clF}.$$
 Moreover,
for every $\Al$-formula $\vf$ we have:
$$\vf \text{ is satisfiable in }  \LSuma{a}{\NPO}{\clF} \tiff \vf \text{ is satisfiable in }  \LSuma{a}{\Tr(\lvf,\lvf)}{\clF}.$$
\item \label{thm:smll-tree-for-Csat-itemCSat}
Assume that $\clI$ is closed under finite disjoint unions. Then
$$\LSuma{a}{\NPO}{\clF} \equiv \LSuma{a}{\clI}{\clF} \equiv  \Diskl{\omega}{\LSuma{a}{\Trf}{\clF}}.$$ Moreover,
for every $\Al$-tie $\tau=(\vf,\Phi,\oGamma)$ we have:
\begin{equation}\label{eq:main}
\tau \text{ is satisfiable in }  \LSuma{a}{\NPO}{\clF} \tiff \tau \text{ is satisfiable in }  {\Disk{\lvf}\LSuma{a}{\Tr(\lvf,\lvf)}{\clF}}.
\end{equation}
\end{enumerate}
\end{theorem}
In view of Proposition \ref{prop:sat_via_cSat}, the first statement
of the theorem is a corollary of the second statement. For the key equivalence (\ref{eq:main}), see
Theorem 5.2(i) in  \cite{AiML2018-sums}.

Theorem \ref{thm:smll-tree-for-Csat} will be the crucial semantic tool for the complexity results.

\section{Complexity}\label{sec:compl}
The main goal of this section is to show that the modal satisfiability problem on sums over Noetherian orders
is polynomial space Turing reducible to the modal satisfiability problem on summands.

For problems $A$ and $B$,  we put $A\RedPSP B$ if
there exists a polynomial space bounded oracle deterministic machine $M$ with oracle $B$ that decides $A$
\cite{SimonGill1977} (it is assumed that every tape of $M$, including the oracle tape, is polynomial space bounded).

\begin{theorem}\label{thm:main_compl_res}
Let $a <\AlM<\omega$, $\clF$ a class of $\AlM$-frames, and
  $\clI$  a class of Noetherian orders containing all finite trees.
  Then:
  \begin{enumerate}
\item   \label{thmItem:main_compl_res-1}
$\Sat{\LSuma{a}{\clI}{\clF}} \RedPSP   \Sat{\univ{\clF}}$.
\item   \label{thmItem:main_compl_res-2}
If also $\clI$ is closed under finite disjoint unions, then
$\Sat{\univ{(\LSuma{a}{\clI}{\clF})}} \RedPSP    \Sat{\univ{\clF}}$.
\end{enumerate}
\end{theorem}

This theorem will be proven in Section \ref{sec:proofOfMain}.
For technical reasons, first we will address complexity of the {\em conditional satisfiability problem}.
Let $\Al$ be finite, and let
$\clF$ be a class of $\Al$-frames.
We shall be interested in whether a given tie  $(\vf,\Phi,\oGamma)$
is satisfiable in $\clF$.
The following simple observation
shows that, w.l.o.g.,
we may assume that
every $\oGamma(a)$, $a<\Al$,
consists of subformulas of $\vf$, and hence that $\oGamma$
is a finite sequence of finite sets:

\begin{proposition}\label{psp:prop:SubfCond}
A tie $(\vf,\Phi,\oGamma)$ is satisfiable in a class $\clF$ iff
$(\vf,\Phi,(\SF{\vf}\cap \oGamma(a))_{a\in \Al})$
is.
\end{proposition}
\begin{proof}
It is immediate from Definition \ref{def:sat-mod} that for any conditions $\oGamma, \oDelta$ such that
\begin{center}
$\oGamma(a)\cap\SF{\vf}=\oDelta(a)\cap\SF{\vf}$ for all $a\in \AlM$,
\end{center}
we have
$$\cmo{\mM}{w}{\oGamma}{\chi} \tiff  \cmo{\mM}{w}{\oDelta}{\chi}$$
for every model $\mM$, every $w$ in $\mM$, and every $\chi\in \SF{\vf}$.

The statement of the proposition is a particular case of this observation, where
$\oDelta(a)=\SF{\vf}\cap \oGamma(a)$ for all $a\in \Al$.
\end{proof}

The
{\em conditional satisfiability problem on $\clF$} is to decide
whether
a
given tie $(\vf,\Phi,\oGamma)$
such that
$\oGamma \subseteq \SF{\vf}^\Al$ is
satisfiable in $\clF$.
In symbols, $\CSat{\clF}$ is the set
$$
\{(\vf,\Phi,\oGamma)\mid
\vf \in \ML(\AlA) \;\&\;
\Phi\subseteq \SF{\vf}\;\&\;
\oGamma\subseteq \SF{\vf}^\Al\;\&\;
\text{ the tie }
(\vf,\Phi,\oGamma)
 \text{ is satisfiable in } \clF\}.
$$

In Section \ref{subs:arithm} we will
describe a decision procedure for the conditional satisfiability problem on sums over Noetherian orders $\LSuma{a}{\NPO}{\clF}$
with the oracle $\CSat{\clF}$. Next, in Section \ref{Sec:SatviaCSat}, we will describe reductions between
$\CSat{\clF}$ and $\Sat{\univ{\clF}}$, which will complete the proof of Theorem \ref{thm:main_compl_res}.

\improve{

\improve{``or every frame in $\clI$ is rooted'' - as a remark}

\todo{Move it:
\comm{Somewhere - when $$\Sat{\univ{\clF}}$$ is equivalent to $$\Sat{\clF}$$}

\begin{corollary}
Let $\clF$ be a non-empty class of $\AlM$-frames such that ..., $a \in \AlM$, and let $\Trf\subseteq \clI \subseteq   \NPO$.
Then $\Sat{\LSuma{a}{\clI}{\clF}} \in \PSPACE( \Sat{\clF} )$.
If also $\clI$ is closed under finite disjoint sums \comm{Exclude?: or every frame in $\clI$ is rooted}, then
$\Sat{\LSuma{a}{\clI}{\clF}} \in \PSPACE( \Sat{\clF})$.
\end{corollary}
}

\improve{
\begin{corollary}
\comm{iterated}
\end{corollary}
}

\improve{Words: main technical tool (CSat-procedure)}

\todo{Illustrations} The classes ..... satisfy the assumptions of the theorem.

\comm{$\zervf?$ $(0)_{\lvf\times\alpha}$}

\todo{The place for the CSat and Proposition? No, later, with the formal notation for vectors.}

\comm{Keeps more information; universal modality? equivalences from AiML? }

}
\subsection{Arithmetic of conditional satisfiability}\label{subs:arithm}


It will be convenient to
encode subformulas of a given $\vf$ as Boolean vectors of length $\lvf$,
considered as characteristic functions on $\SF{\vf}$.
For $\Al\leq \omega$, the set
of modal formulas
is linearly ordered by a polynomial time
computable
relation $\ordforms$ such
that if $\psi$ is a subformula of $\vf$,
then $\vf \ordforms \psi$
(e.g.,
put
$\vf \ordforms \psi$ if
$\psi$ is shorter than $\vf$,
and assume that formulas of the same length are ordered lexicographically).
Let $(\psi_0,\ldots,\psi_{\lvf-1})$ be the $\ordforms$-chain of all subformulas of $\vf$ (hence, $\psi_0=\vf$);
for $\vect\in 2^\lvf$, we write $\fv{\vf}{\vect}$ for $\{\psi_i\mid \vect(i)=1\}$;
similarly, a sequence $\vectU=(\vectu_a)_{a\in\Al}$ of such vectors represents
the condition $\oGamma=(\fv{\vf}{\vectu_a})_{a\in\Al}$.
Hence, for a finite $\Al$, every tie $\tau=(\vf,\Phi,\oGamma)$ with $\oGamma \subseteq \SF{\vf}^\Al$ is
represented by a triple $\tau'=(\vf,\vect,(\vectu_a)_{a\in \Al})$, where $\vect,\vectu_0,\ldots \vectu_{\Al-1}\in 2^\lvf$;
this triple is also called a {\em tie}. In this case, by the satisfiability of $\tau'$ we mean the satisfiability of $\tau$.
\improve{Formally, $\tau'\in\CSat\clF$ in undefined.}

Let $\zervf$ denote the sequence of length $\Al$ of zero vectors of length  $\lvf$ (that is, $\zervf$ represents the condition,
consisting of empty sets).
In view of Proposition \ref{psp:prop:SubfCond}, we have
the following reformulation of Proposition \ref{prop:sat_via_cSat}:
\begin{proposition}\label{prop:sat_via_cSat-vect}
$\vf$ is satisfiable in $\clF$ iff there exists $\vect\in2^\lvf$
 such that $\vect(0)=1$ and the tie $(\vf,\vect,\zervf)$
is satisfiable in $\clF$.
\end{proposition}


\hide{
Let $\bigsqcup \clF$ be the class of frames of form $\bigsqcup_I \frF_i$, where
$I$ is a non-empty set and
$\frF_i\in \clF$ for all $i\in I$, and let
$\DisFin{\clF}$ be
класс таких шкал при конечных $I$,
и $\Disk{k}{\clF}$  --- класс таких шкал при $I$ мощности не более $k$.
}

For Boolean vectors
$\vect=(v_0,\ldots, v_{l-1})$, $\vectu=(u_0,\ldots, u_{l-1})$,  let $\vect+\vectu$ be their element-wise disjunction
$(\max\{v_0,u_0\},\ldots, \max\{v_{l-1},u_{l-1}\})$.

\begin{lemma}\label{psp:lem:satsumm-decomp-wide-poly}
Let $\clG$ be a class of  $\AlM$-frames and $0<b< \omega$. A tie $(\vf,\vect,\vectU)$ is satisfiable  in $\Disk{b}{\clG}$
iff there exist a positive $k\leq b$, $\vect_0,\dts,\vect_{k-1}\in 2^{\lvf}$ such that
$\vect=\sum_{i<k}\vect_i$ and for every $i<k$ the tie $(\vf,\vect_i, \vectU)$  is satisfiable in $\clG$.
\end{lemma}
\begin{proof}
This proposition is a corollary of Proposition \ref{prop:truth-pres-for-cs} formulated in our new vector notation for ties. Indeed,
by Proposition \ref{prop:truth-pres-for-cs}, we have
\begin{equation}\label{eq:satsumm-decomp-vect-1-dis}
\SFMG{\vf}{\bigsqcup_{i<k}\mM_i}{\oGamma}=\bigcup_{i<k}\SFMG{\vf}{\mM_i}{\oGamma}
\end{equation}
for every $\AlM$-models $\mM_0,\dts,\mM_{k-1}$  and every $\AlM$-condition $\oGamma$.
We are interested in the situation when $0<k\leq b$ and frames of $\mM_0,\dts,\mM_{k-1}$  are in $\clG$.
Assuming that $\oGamma$  is the condition represented
by $\vectU$,
and vectors $\vect, \vect_1,\ldots, \vect_{k-1}$ are given by
the identities $\fv{\vf}{\vect}=\SFMG{\vf}{\bigsqcup_{i<k}\mM_i}{\oGamma}$,
$\fv{\vf}{\vect_0} =\SFMG{\vf}{\mM_0}{\oGamma}$, \ldots, $\fv{\vf}{\vect_{k-1}} =\SFMG{\vf}{\mM_{k-1}}{\oGamma}$,
we see that
(\ref{eq:satsumm-decomp-vect-1-dis}) takes the form $\vect=\sum_{i<k}\vect_i$.
\end{proof}


For Boolean vectors  $\vect, \vectu_0,\ldots,\vectu_{\AlM-1}$  of the same length and $a\in\AlM$, we put
 $(\vectu_0,\ldots,\vectu_{\Al-1})\plusa\vect=(\vectu'_0,\ldots,\vectu'_{\Al-1})$,
where $\vectu'_a=\vectu_a+ \vect$, and
$\vectu'_b=\vectu_b$ for $b\neq a$.

Similarly to models, for $\AlM$-frames $\frF_0$ and $\frF_1$ and $a\in \Al$  we define $\frF_0\plusa\frF_1$ as $\LSuma{a}{(2,<)}{\frF_i}$;
for classes $\clF$ and $\clG$  of $\AlM$-frames, let $\clF\plusa\clG=\{\frF\plusa\frG \mid \frF\in \clF \,\&\, \frG\in\clG\}$.

\begin{lemma}\label{psp:lem:satsumm-decomp-top-poly}
Let $\clF$ and $\clG$ be classes of $\AlM$-frames and $a\in \AlM$. Then a tie
$(\vf,\vect,\vectU)$ is satisfiable in $\clF\plusa\clG$  iff
there exist $\vect_0,\vect_1\in 2^{\lvf}$ such that
\begin{enumerate}
\item $\vect=\vect_0+\vect_1$, and
\item $(\vf,\vect_0,\vectU\plusa\vect_1)$ is satisfiable in $\clF$, and
\item $(\vf,\vect_1,\vectU)$ is satisfiable in $\clG$.
\end{enumerate}
\end{lemma}
\begin{proof}
Let
$\mM_0$ be a model on a frame in $\clF$, and let $\mM_1$ be a model on a frame in $\clG$.
Let
$\oGamma$ be the condition represented by $\vectU$, i.e.,  $\oGamma=(\fv{\vf}{\vectU(a)})_{a\in\Al}$, and let  $\Psi=\SFMG{\vf}{\mM_1}{\oGamma}$.
By Lemma \ref{lem:satsumm-decomp-poly},
\begin{equation}\label{eq:satsumm-decomp-vect-1-pol}
\SFMG{\vf}{\mM_0\plusa\mM_1}{\oGamma}=\SFMG{\vf}{\mM_0}{\oGamma\cupa\Psi}\cup \Psi.
\end{equation}

For the ``only if'' part, assume that $\fv{\vf}{\vect}=\SFMG{\vf}{\mM_0\plusa\mM_1}{\oGamma}$.
Consider tuples  $\vect_0,\vect_1\in2^\lvf$ such that
\begin{eqnarray}
\label{eq:satsumm-decomp-vect-1-pol-2}
\fv{\vf}{\vect_0}&=&\SFMG{\vf}{\mM_0}{\oGamma\cupa\Psi}, \text{ and }\\
\label{eq:satsumm-decomp-vect-1-pol-3}
\fv{\vf}{\vect_1}&=&\SFMG{\vf}{\mM_1}{\oGamma}.
\end{eqnarray}
The identity (\ref{eq:satsumm-decomp-vect-1-pol-2}) says that $(\vf,\vect_0,\vectU\plusa\vect_1)$ is satisfiable in $\clF$, and
the identity (\ref{eq:satsumm-decomp-vect-1-pol-3}) says that $(\vf,\vect_1,\vectU)$ is satisfiable in $\clG$.
Since $\Psi=\SFMG{\vf}{\mM_1}{\oGamma}$, by
(\ref{eq:satsumm-decomp-vect-1-pol}) we obtain $\fv{\vf}{\vect}=\fv{\vf}{\vect_0}\cup\fv{\vf}{\vect_1}$, that is $\vect=\vect_0+\vect_1$.

For the ``if'' part, assume that for some $\vect_0,~ \vect_1$  with $\vect=\vect_0 + \vect_1$  we
have
(\ref{eq:satsumm-decomp-vect-1-pol-2})
(this can be assumed since $(\vf,\vect_0,\vectU\plusa\vect_1)$ is satisfiable in $\clF$), and
(\ref{eq:satsumm-decomp-vect-1-pol-3})
(this can be assumed since $(\vf,\vect_1,\vectU)$ is satisfiable in $\clG$).
Since $\vect=\vect_0+ \vect_1$, we have $\fv{\vf}{\vect}=\fv{\vf}{\vect_0}\cup\fv{\vf}{\vect_1}$.
Now by (\ref{eq:satsumm-decomp-vect-1-pol}) we obtain
that $\fv{\vf}{\vect}=\SFMG{\vf}{\mM_0+\mM_1}{\oGamma}$. The latter proves that
$(\vf,\vect,\vectU)$ is satisfiable in $\clF\plusa\clG$.
\end{proof}

\comm{Key lemma(?)}
\begin{lemma}\label{lem:psp:SatSumTree-poly}
Let $\clF$  be a class of $\AlM$-frame, $a\in \AlM$, and $0<h,b<\omega$.
A tie $(\vf,\vect,\vectU)$ is satisfiable in
$\LSuma{a}{\Tr(h+1,b)}{\clF}$  iff it is satisfiable  in $\clF$, or
there exist a positive  $k\leq b$  and  $\vectu,\vect_0,\dts,\vect_{k-1}\in 2^{\lvf}$ such that
\begin{enumerate}
\item
$\vect=\vectu+ \sum_{i<k}\vect_i$, and
\item
the tie $(\vf,\vectu,\vectU\plusa\sum_{i<k}\vect_i)$ is satisfiable in $\clF$, and
\item for all $i<k$, the tie $(\vf,\vect_i,\vectU)$ is satisfiable in $\LSuma{a}{\Tr(h,b)}{\clF}$.
\end{enumerate}
\end{lemma}
\begin{proof}
By the definition of $\LSuma{a}{\Tr(h+1,b)}{\clF}$,
the tie  $\tau=(\vf,\vect,\vectU)$ is satisfiable in
$\LSuma{a}{\Tr(h+1,b)}{\clF}$ iff $\tau$ is satisfiable in $\clF$ or
in $\clF\plusa\clG$, where $\clG=\Disk{b}{\LSuma{a}{\Tr(h,b)}{\clF}}$.

By Lemma \ref{psp:lem:satsumm-decomp-top-poly}, $\tau$  is satisifable in   $\clF\plusa\clG$  iff there exist
$\vectu,\vectu'\in 2^{\lvf}$ such that
$\vect=\vectu+\vectu'$,
$(\vf,\vectu,\vectU\plusa\vectu')$ is satisfiable in $\clF$, and
$(\vf,\vectu',\vectU)$ is satisfiable in $\clG$.
By Lemma  \ref{psp:lem:satsumm-decomp-wide-poly}, $(\vf,\vectu',\vectU)$ is satisfiable in  $\clG$
iff  there exist
 $0<k\leq b$ and
 $\vect_0,\dts,\vect_{k-1}\in 2^{\lvf}$ such that
$\vectu'=\sum_{i<k}\vect_i$, and for all  $i<k$ the tie $(\vf,\vect_i,\vectU)$  is satisfiable in  $\LSuma{a}{\Tr(h,b)}{\clF}$.
\end{proof}

This lemma
allows to describe   the  procedure  $\CondSatSumFa$ (Algorithm \ref{alg:SatSumFa-poly})
which using an oracle for\improve{\todo{for?}} $\CSat{\clF}$ decides whether a
given tie is satisfiable
in
$\LSuma{a}{\Tr(h,b)}{\clF}$.
Namely, we have:
\begin{theorem}\label{Thm:CorrectAlgTree}
Let $a<\Al<\omega$,
 $\clF$ a class of $\Al$-frames,
 $(\vf,\vect,\vectU)$ an $\Al$-tie, and $0<h,b<\omega$. Then
 $$
 (\vf,\vect,\vectU) \text{ is satisfiable in }
\LSuma{a}{\Tr(h,b)}{\clF} \tiff  \CondSatSumFa(\vf,\vect,\vectU,h,b)
\text{ returns }\true.
$$
\hide{
 $$
 $(\vf,\vect,\vectU)$ is satisfiable in
$\CSat \LSuma{a}{\Tr(h,b)}{\clF}$ iff  $\CondSatSumFa(\vf,\vect,\vectU,h,b)$
(Algorithm \ref{alg:SatSumFa-poly})
returns $\true$.
}
\end{theorem}
This theorem will be our main technical tool for complexity results.
For the first of its corollaries, we show how
to reduce the conditional satisfiability on sums to the conditional satisfiability on summands.
\improve{Eng}
\begin{theorem}\label{thm:main_compl_resCSat}
Let $a <\AlM<\omega$, $\clF$ a class of $\AlM$-frames, and
  $\clI$  a class of Noetherian orders containing all finite trees.
  Then:
  \begin{enumerate}
\item \label{thm:main_compl_resCSat_i1}
$\Sat{\LSuma{a}{\clI}{\clF}} \RedPSP   \CSat{\clF}$.
\item   \label{thm:main_compl_resCSat_i2}
If also $\clI$ is closed under finite disjoint unions, then
$\CSat{\LSuma{a}{\clI}{\clF}} \RedPSP    \CSat{\clF}$.
\end{enumerate}
\end{theorem}
\improve{
\todo{Somewhere, after the proof of the main Theorem, put a remark on rooted case; say, why it is so important (and confess about the annotated formulation in AiML)
}
}
\improve{
\todo{The proof of this version is based on Lemma 11 [WRONG NUMBER] in \cite{AiML2018-sums} (in Diser it was proved in more details). Exclude. Give a comment. And give a comment (just before publishing) for erratum in the announced in \cite{AiML2018-sums} formulation. }
}
\comm{A simplified version:

\begin{theorem}\label{thm:CSatPspace}\comm{$\AlM<\omega$?}
Let $a \in \AlM<\omega$ and $\clF$ a class of $\AlM$-frames.
If $\clI$ is a class of Noetherian orders containing $\Trf$ and
closed under finite disjoint sums, 
then 
$\CSat{\LSuma{a}{\clI}{\clF}}\in \PSPACE(\CSat \clF).$
\end{theorem}
}
\begin{proof}
By Theorem \ref{thm:smll-tree-for-Csat}(\ref{thm:smll-tree-for-Csat-itemSat}), a formula
$\vf$
      is satisfiable in
      $\LSuma{a}{\clI}{\clF}$
iff
$\vf$  is satisfiable in  $\LSuma{a}{\Tr(\lvf,\lvf)}{\clF}$, and by
Proposition \ref{prop:sat_via_cSat-vect}, this means that there exists
a satisfiable in
$\LSuma{a}{\Tr(\lvf,\lvf)}{\clF}$
tie $(\vf,\vect,\zervf)$ with $\vect(0)=1$.
From
Theorem
\ref{Thm:CorrectAlgTree} we obtain
\begin{lemma}\label{Lem:nnSat}
A formula
      \improve{$\Al$-formula, $\Al$-tie -- see next item}
      $\vf$
      is satisfiable in
      $\LSuma{a}{\clI}{\clF}$
      iff
      there exists
      $\vect\in  2^\lvf$ such that $\vect(0)=1$
      and  
      $\CondSatSumFa(\vf,\vect,\zervf,\lvf,\lvf)$
      returns $\true$.
\end{lemma}
Assume that $\clI$ is closed under finite disjoint unions.
By  Theorem \ref{thm:smll-tree-for-Csat}(\ref{thm:smll-tree-for-Csat-itemCSat}),
a tie  $(\vf,\vect,\vectU)$ is satisfiable in
 $\LSuma{a}{\clI}{\clF}$
 iff
    it is satisfiable in   ${\Disk{\lvf}\LSuma{a}{\Tr(\lvf,\lvf)}{\clF}}$.   By  Lemma
    \ref{psp:lem:satsumm-decomp-wide-poly}, this means that
    there exists $k\leq \lvf$ and tuples $\vect^{(0)},\dts,\vect^{(k-1)}\in 2^\lvf$ such that
$\vectu=\sum_{i<k}\vect^{(i)}$ and
$(\vf,\vect^{(i)},\vectU,\lvf,\lvf)$
is satisfiable in
$\LSuma{a}{\Tr(\lvf,\lvf)}{\clF}$
 for every $i<k$.
Using Theorem
\ref{Thm:CorrectAlgTree} again, we obtain

\begin{lemma}\label{Lem:nnCSat}
      If
       $\clI$ is closed under finite disjoint unions,
       then a tie
       $(\vf,\vect,\vectU)$
       is satisfiable in
      $\LSuma{a}{\clI}{\clF}$
      iff
        there exist $k\leq \lvf$ and tuples $\vect^{(0)},\dts,\vect^{(k-1)}\in 2^\lvf$ such that
$\vectu=\sum_{i<k}\vect^{(i)}$ and
$\CondSatSumFa(\vf,\vect^{(i)},\vectU,\lvf,\lvf)$  returns $\true$  for every $i<k$.
\end{lemma}

Set $n=\lvf$. Let us estimate the amount of space used by $\CondSatSumFa$ for the case $n=h=b$.
At each call $\CondSatSumFa$ needs $\BigO(n^2)$ space
to store new variables.
The depth of recursion is bounded by $n$.
Thus, 
we can reduce
$\Sat{\LSuma{a}{\clI}{\clF}}$ to the conditional satisfiability problem on $\clF$ in $\BigO(n^3)$ space by  Lemma \ref{Lem:nnSat}. This proves the first statement of the theorem.
In the case when   $\clI$ is closed under finite disjoint unions,
$\CSat{\LSuma{a}{\clI}{\clF}}$ is also reducible to the conditional satisfiability problem on $\clF$ in $\BigO(n^3)$ space by Lemma \ref{Lem:nnCSat}. This proves the second statement of the theorem.
\end{proof}

\begin{algorithm}[t]
\SetAlgoNoLine
\caption{Decision procedure for $\CSat \LSuma{a}{\Tr(h,b)}{\clF}$
with an oracle for $\CSat{\clF}$}
\label{alg:SatSumFa-poly}

\SetKwFunction{mCsat}{$\CondSatSumFa$}
\mCsat{$\vf,\vect,\vectU,h,b$}:${boolean}$
\\ \Indp
\KwIn{A tie $(\vf,\vect,\vectU)$; positive integers $h,b$}

\lIf*{$(\vf,\vect,\vectU)$ is satisfiable in $\clF$}{ \Return ${\true}$\;}
{
\lIf*{$h>1$}
{
\\ \Indp
\lFor*{$k$ such that $1\leq k\leq b$}
{
\lFor*{$\vectu,\vect_0,\dts,\vect_{k-1}\in 2^{\lvf}$ such that $\vect=\vectu+ \sum_{i<k} \vect_{i}$}
{
\\ \Indp
\lIf*{$(\vf,\vectu, \vectU\plusa  \sum_{i<k} \vect_{i})$ is satisfiable in $\clF$}
{
\\ \Indp
\lIf*{$\bigwedge\limits_{i<k}\CondSatSumFa(\vf,\vect_i,\vectU,h-1,b)$}
{
\Return ${\true}$\;
}
}
}
}
}
}
\Return ${\false}$.
\end{algorithm}

\hide{

\begin{algorithmic}
\Require  формула $\vf$, булев вектор  $\vect$ длины $\lvf$, кортеж ${\vectU=(\vectu_0,\ldots,\vectu_{N-1})}$, состоящий из булевых векторов длины $\lvf$,
$h>0$, $b\geq 0$,
\Function{$\CondSatSumFa$}{$\vf$,$\vect$,$\vectU$, $h$,$b$}:{boolean}
\If{$\CondSatF(\vf,\vect,\vectU)$} \Return \textbf{True}
\EndIf
\If{$h>1$}
\ForAll{$k$ such that $1\leq k\leq b$}
\ForAll{$\vectu,\vect_0,\dts,\vect_{k-1}\in 2^{\lvf}$ such that
$\vect=\vectu+ \vect_0 +\dts +\vect_{k-1}$
 }
\If{
$\CondSatF(\vf,\vectu, \vectU\plusa(\vect_0 +\dts +\vect_{k-1}))$
}
\If{$\bigwedge\limits_{i<k}\CondSatSumFa(\vf,\vect_i,\vectU,h-1,b)$}
\State\textbf{return True}
\EndIf
\EndIf
\EndFor
\EndFor
\EndIf
\State\textbf{return False}
\EndFunction
\end{algorithmic}
\end{algorithm}
}

\subsection{Reductions between  $\Sat{\clF}$, $\CSat{\clF}$, and $\Sat{\univ{\clF}}$}\label{Sec:SatviaCSat}
By Proposition \ref{prop:sat_via_cSat-vect},
the satisfiability problem on a class $\clF$ is
polynomial space Turing reducible to the conditional satisfiability problem on $\clF$.
Let $\RedKarp$ denote the polynomial-time many-to-one reduction.
Below we show that in many cases $\CSat{\clF}$ is polynomial time reducible to $\Sat{\clF}$, that is
$\CSat{\clF}\RedKarp \Sat{\clF}$;
hence, in these cases,
the above two problems are equivalent with respect to $\RedPSP$.
This fact is close to a result
obtained in  \cite{SpaanUniv}, where it was shown that in many situations there exists
a (stronger than $\RedPSP$, but weaker\footnote{I am using ``stronger'' and ``weaker'' in a non-strict sense.}  than $\RedKarp$) reduction of $\Sat{\univ{\clF}}$ to $\Sat{\clF}$.
Let us discuss reductions between the above two problems in more details.
\improve{Reread; what is the role of $\Sat{\univ{\clF}}$?}

For a binary relation $R$ on a set $W$, put $R^{\leq m}=  \bigcup_{n \leq m} R^n$, where $R^0$ is the identity relation on $W$, $R^{n+1}=R\circ R^n$
($\circ$ denotes the composition of relations).
Recall that  $R^*$ denotes the transitive reflexive closure of $R$: $R^*=\bigcup_{n <\omega} R^n$.
$R$ is said to be {\em $m$-transitive} if $R^{\leq m}$ includes $R^{m+1}$,  or equivalently, $R^{\leq m}=R^*$  (e.g., every transitive relation is 1-transitive).
For a frame $\frF=(W,(R_a)_{a\in \Al})$, put $R_\frF=\bigcup_{a\in \Al} R_a$.
We say that $\frF$
is {\em $m$-transitive} if the relation $R_\frF$ is.
In particular, if one of the relations of $\frF$ is universal (i.e., is equal to $W\times W$), then
 $\frF$ is 1-transitive.

For $w$ in $\frF$, let $\cone{\frF}{w}$ denote the subframe of $\frF$ generated by the singleton $\{w\}$ (that is,
$\cone{\frF}{w}$ is the restriction of $\frF$ on the set $\{v\mid w R^*_\frF v\}$); such frames are called {\em cones}).
A class $\clF$ of frames is {\em closed under taking cones} if
for every $\frF$ in $\clF$  and for every $w$ in $\frF$, the cone $\cone{\frF}{w}$ is in $\clF$.

A class $\clF$ of $\AlA$-frames is said to be  {\em preconical},  \improve{\todo{improve: a better term?}}
if 
\begin{itemize}
\item $\clF$ is closed under taking cones, and
\item there exists $m<\omega$ such that every frame in $\clF$ is $m$-transitive, and
\item
for every $\frF$ in $\clF$,
the relation $R_\frF^*$ is downward directed (i.e.,
for every $w,v$ in $\frF$ there exists $u$
such that $u R_\frF^* w$ and $u R_\frF^* v$).
\end{itemize}

\begin{proposition}\label{prop:CsatToSat-precon}
Let $\Al$ be finite and $\clF$ a class of $\Al$-frames.
If $\clF$ is preconical, then
 $\CSat{\clF}\RedKarp\Sat{\clF}$.
\end{proposition}
\begin{proof}

Given a condition $\oGamma$ and a formula $\vf$,
we define the formula $\trc{\vf}{\oGamma}$ as follows:
${\trc{\bot}{\oGamma}=\bot}$,  $\trc{p}{\oGamma}=p$ for variables, $\trc{\vf_1\imp\vf_2}{\oGamma} =\trc{\vf_1}{\oGamma}\imp\trc{\vf_2}{\oGamma}$, and
\[
\trc{\Di_a \vf}{\oGamma}=\left\{
\begin{array}{ll}
\top,& \text{ if } \vf\in \oGamma(a), \\
\Di_a\trc{\vf}{\oGamma} &\text{ otherwise.}
\end{array}\right.
\]
For every $\Al$-model $\mM$, we have:
\begin{equation}\label{eq:gammaTr-model}
\cmo{\mM}{w}{\oGamma}{\vf} \;\tiff\; \mM,w\mo\trc{\vf}{\oGamma}
\end{equation}
The proof is straightforward, see \cite[Lemma 4.7]{AiML2018-sums} for the details.

Let $\Di\vf$ abbreviate the $\Al$-formula $\bigvee_{a\in \Al} \Di_a \vf$, and let
$\Di^0\vf =\vf$, $\Di^{m+1}\vf=\Di\Di^m\vf$,
$\Di^{\leq m}\vf=\bigvee_{n\leq m} \Di^n \vf$.

For a tie $(\vf,\vect, \vectU)$ with $\vectU=(\vectu_a)_{a\in\Al}$, we put
\begin{equation}
\delta_m(\vf,\vect,\vectU)=
\bigwedge_{\psi\in \fv{\vf}{\vect}} \Di^{\leq m} \trc{\psi}{\oGamma}
\con
\bigwedge_{\psi\in \SF{\vf}{\setminus}\fv{\vf}{\vect}}  \neg \Di^{\leq m} \trc{\psi}{\oGamma},
\end{equation}
where $\oGamma$ is the condition represented by $\vectU$, i.e., $\oGamma=(\fv{\vf}{\vectu_a})_{a\in\Al}$.

Since $\clF$ is preconical, there exists a finite $m$ such that every frame in $\clF$ is $m$-transitive. We claim that
\begin{equation}
(\vf,\vect,\vectU)\in \CSat{\clF} \quad \tiff \quad \delta_m(\vf,\vect,\vectU)\in \Sat{\clF}.
\end{equation}

First, assume that $(\vf,\vect,\vectU)$ is satisfiable in a frame $\frF\in\clF$. This means that
for a model $\mM$ based on $\frF$ we have $\fv{\vf}{\vect}=\SFMG{\vf}{\mM}{\oGamma}$,
where $\oGamma$ is the condition represented by $\vectU$.
For every $\psi\in \fv{\vf}{\vect}$ we choose a point $w_\psi$ such that $\cmo{\mM}{w_\psi}{\oGamma}{\vf}$, and then put
$V=\{w_\psi\mid \psi\in \fv{\vf}{\vect} \}$.
The relation $R_\frF^*$ is downward directed, hence there exists a point $w$ in $\mM$ such that
$w R_\frF^* v$ for all $v$ in $V$; by $m$-transitivity, $w R_\frF^{\leq m} v$.
It follows that if $\psi\in \fv{\vf}{\vect}$, then ${\mM,w \mo\Di^{\leq m} \trc{\psi}{\oGamma}}$: indeed,
we have
${\mM,w_\psi \mo \trc{\psi}{\oGamma}}$ by (\ref{eq:gammaTr-model}) and $w R_\frF^{\leq m} w_\psi$.
On the other hand, if a subformula $\psi$ of $\vf$ is not in $\fv{\vf}{\vect}$, then
 $\trc{\psi}{\oGamma}$ is false at every point in $\mM$ by (\ref{eq:gammaTr-model}), and so
$\mM,w\mo \neg \Di^{\leq m}\trc{\psi}{\oGamma}$.
It follows that
the formula $\delta_m(\vf,\vect,\vectU)$ is satisfiable in $\clF$.\improve{\todo{Read}}

Now assume that $\delta_m(\vf,\vect,\vectU)$ is true at a point $w$ in a model $\mM$ over a frame $\frF\in\clF$.
Since $\clF$ is closed under taking cones, we may assume that $\frF=\cone{\frF}{w}$ (recall that
if a formula is true at a point in a model, then it is true at this point in the model generated by this point).
Let $\psi$ be a subformula of $\vf$. Suppose that $\psi\in \fv{\vf}{\vect}$.
Then $\mM,w\mo \Di^{\leq m}\trc{\psi}{\oGamma}$. Hence,
the formula $\trc{\psi}{\oGamma}$ is true at a point $u$ of $\mM$, which means that
$\cmo{\mM}{u}{\oGamma}{\psi}$
by (\ref{eq:gammaTr-model}). Thus,  $\psi\in \SFMG{\vf}{\mM}{\oGamma}$.
On the other hand, if $\psi\nin \fv{\vf}{\vect}$, then
$\mM,w\mo \neg \Di^{\leq m} \trc{\psi}{\oGamma}$; by $m$-transitivity, we obtain
that $\trc{\psi}{\oGamma}$ is false at every point of $\mM$; using (\ref{eq:gammaTr-model}) again, we obtain
$\psi\nin \SFMG{\vf}{\mM}{\oGamma}$. Thus, $\fv{\vf}{\vect}=\SFMG{\vf}{\mM}{\oGamma}$.
Since $\oGamma=(\fv{\vf}{\vectu_a})_{a\in\Al}$,  the tie
$(\vf,\vect,\vectU)$ is satisfiable in $\clF$.
\end{proof}

\begin{proposition}\label{prop:CsatToSatUniv}
If $\Al$ is finite and $\clF$ is a class of $\Al$-frames, then
$\CSat{\clF}\RedKarp\Sat{\univ{\clF}}$.
\end{proposition}
\begin{proof}
It is trivial that $\CSat{\clF}\RedKarp\univ{\CSat{\clF}}$
(the reduction increases by one the indexes of modalities occurring in formulas of a given tie).

The class $\univ{\clF}$ is preconical: if $\frF\in\clF$, then $\frF$ is 1-transitive,
$\cone{\frF}{w}=\frF$ for every $w\in \frF$, and $R_\frF^*$ is downward-directed, since $R_\frF$ is the universal relation on $\frF$.
Hence,  $\CSat{\univ{\clF}}\RedKarp \Sat{\univ{\clF}}$ by Proposition \ref{prop:CsatToSat-precon}.
\end{proof}

Now let us describe a reduction of $\Sat{\univ{\clF}}$ to $\CSat{\clF}$.
This reduction is based on the following construction proposed in
\cite{SpaanUniv}.
Let $\vf$ be a formula in the language $\ML(1+\Al)$.
For subformulas $\Di_0\psi$ of $\vf$ starting with $\Di_0$, we choose distinct variables  $p_\psi$ not occurring in $\vf$,
and put for subformulas of $\vf$:
$\trS{\bot}=\bot$;  $\trS{p}=p$ for variables; $\trS{(\vf_1\imp\vf_2)}=\trS{\vf_1}\imp\trS{\vf_2}$;
$\trS{(\Di_a\psi)}=\Di_a\trS{\psi}$ for $a>0$; and
$\trS{(\Di_0\psi)}=p_{\psi}$.
We have for a class $\clF$ of $\Al$-frames:
\begin{equation}\label{eq:spaan-univ}
\vf\in\Sat{\univ{\clF}} \tiff
\trS{\vf}\wedge \bigwedge_{\Di_0\psi\in \SF{\vf}} \left(
(\Di_0\trS{\psi}\leftrightarrow \Box_0 p_{\psi})
\con (\Box_0 p_\psi \vee \Box_0 \neg p_\psi)
\right) \in \Sat{\univ{\clF}},
\end{equation}
see \cite[Lemma 4.5]{SpaanUniv} for details.\footnote{
The formula in (\ref{eq:spaan-univ}) is an equivalent form of the formula $\vf_{flat}$ used in
\cite[Lemma 4.5]{SpaanUniv}.}

\improve{\todo{No one cares about the difference between this formulation and the one given by Spaan (in \cite{SpaanUniv}, ``closure'' is used). The above is more similar to my AiML06).
So double check for yourself and forget about this boring step.
}}

Let $\eta$ denote the formula in the right-hand side of 
the above equivalence.
Let us show how the satisfiability of $\eta$ in  $\univ{\clF}$ can be expressed as the satisfiability of a tie in $\clF$.
Consider the formula
$$\xi_\vf=\trS{\vf} \con \bigwedge_{\Di_0\psi\in\SF{\vf}}(\trS{\psi} \wedge \neg p_\psi)$$
(the only role of the second conjunct is to have
$\trS{\psi}$, $\neg p_\psi$, $p_\psi$ that occur in $\eta$ as subformulas of $\xi_\vf$).
Let $\mM$ be a model   on a frame in $\univ{\clF}$, and let
$\Phi$ be the set of
subformulas of $\xi_\vf$ that are satisfiable in $\mM$, that is, $\Phi=\SFMG{\xi_\vf}{\mM}{(\emp)_{1+\Al}}$.
Then $\eta$ is true at a point in $\mM$ iff
\begin{equation}\label{eq:spaan-to-tie}
\trS{\vf}\in \Phi \text{ and for every }\psi \text{ with }  \Di_0\psi\in\SF{\vf},~\trS{\psi}\in \Phi \tiff  p_\psi\in \Phi \tiff \neg p_\psi\notin \Phi.
\end{equation}
\todo{DC}
It follows that
$\eta$ is satisfiable in $\univ{\clF}$ iff there exists
$\Phi\subseteq\SF{\xi_\vf}$ satisfying (\ref{eq:spaan-to-tie}) and the tie $(\xi_\vf, \Phi,(\emp)_{1+\Al})$ is satisfiable in $\univ{\clF}$.
The formula $\xi_\vf$ and so, the formulas in $\Phi$, do not contain $\Di_0$.
Hence, the satisfiability of  $(\xi_\vf, \Phi,(\emp)_{1+\Al})$ in $\univ{\clF}$ is equivalent to
the satisfiability of $(\trDec{\xi_\vf}, \trDec{\Phi},(\emp)_{\Al})$ in $\clF$,
where $\trDec{\xi_\vf}$ and $\trDec{\Phi}$ are obtained by decreasing indexes of modalities by 1.
\hide{
Putting everything together,
we have proven
\begin{proposition}\label{prop:satU_to_CSat}
Let $\Al$ be finite and $\clF$ a class of $\Al$-frames. Then
$$\vf\in\Sat{\univ{\clF}} \tiff \text{ there exists }
\Phi\subseteq\SF{\xi_\vf}  \text{ satisfying (\ref{eq:spaan-to-tie}) } \text{ such that }
(\trDec{\xi}, \trDec{\Phi},(\emp)_{\Al})  \text{ is satisfiable in }  \clF.
$$
\end{proposition}
}
Putting everything together,
we have shown that
$$\vf\in\Sat{\univ{\clF}} \tiff \text{ there exists }
\Phi\subseteq\SF{\xi_\vf}  \text{ satisfying (\ref{eq:spaan-to-tie}) } \text{ s. t.  }
(\trDec{\xi}, \trDec{\Phi},(\emp)_{\Al})  \text{ is satisfiable in }  \clF.
$$
This proves
\begin{proposition}\label{prop:satU_to_CSat}
If $\Al$ is finite and $\clF$ is a class of $\Al$-frames, then
$\Sat{\univ{\clF}} \RedPSP \CSat{\clF}$.
\end{proposition}

\begin{remark}
In fact,  Proposition \ref{prop:satU_to_CSat} provides a stronger
reduction than $\RedPSP$.
\end{remark}

 \improve{mention Undecidability}

\subsection{Proof of Theorem \ref{thm:main_compl_res}}\label{sec:proofOfMain}

In view of the above propositions, Theorem \ref{thm:main_compl_res} is an easy corollary of Theorem \ref{thm:main_compl_resCSat}.
\begin{proof}[Proof of Theorem \ref{thm:main_compl_res}]
We obtain $\Sat{\LSuma{a}{\clI}{\clF}} \RedPSP \CSat{\clF}$ by Theorem \ref{thm:main_compl_resCSat}(\ref{thm:main_compl_resCSat_i1}),
and then \mbox{$\CSat{\clF} \RedKarp \Sat{\univ{\clF}}$}
 by Proposition \ref{prop:CsatToSatUniv}. Since $\RedKarp$ is stronger than $\RedPSP$,
 we obtain $\CSat{\clF} \RedPSP \Sat{\univ{\clF}}$. Hence, \mbox{$\Sat{\LSuma{a}{\clI}{\clF}}  \RedPSP \Sat{\univ{\clF}}$},
which proves the first statement of the theorem.

To prove the second statement, we start with
Proposition \ref{prop:satU_to_CSat} and obtain
$\Sat{\univ{(\LSuma{a}{\clI}{\clF})}} \RedPSP \CSat{(\LSuma{a}{\clI}{\clF})}$.
Now by  Theorem \ref{thm:main_compl_resCSat}(\ref{thm:main_compl_resCSat_i2}) and Proposition \ref{prop:CsatToSatUniv}, we obtain
$\CSat{\LSuma{a}{\clI}{\clF}} \RedPSP  \Sat{\univ{\clF}}$.
So $\Sat{\univ{(\LSuma{a}{\clI}{\clF})}} \RedPSP \Sat{\univ{\clF}}$, which completes the proof.
\end{proof}

Recall that in the preconical case, the conditional satisfiability problem is reducable to the standard modal satisfiability problem (Proposition \ref{prop:CsatToSat-precon}), and so $\CSat\clF$, $\Sat\clF$, and
$\Sat{\univ{\clF}}$ are $\RedPSP$-equivalent by Propositions \ref{prop:CsatToSatUniv}, \ref{prop:satU_to_CSat} (in general,
the decidability of $\Sat{\clF}$ does not imply the decidability of $\Sat{\univ{\clF}}$, see \cite[Theorem 4.2.1]{Spaan1993complexity}).
 In this case, Theorem \ref{thm:main_compl_res}
can be reformulated in the following way:
\begin{corollary}\label{thm:main_compl_resConical}
Let $a <\AlM<\omega$, $\clF$ a class of $\AlM$-frames, and
  $\clI$  a class of Noetherian orders containing all finite trees.
If $\clF$ is preconical,  then:
  \begin{enumerate}
\item   $\Sat{\LSuma{a}{\clI}{\clF}} \RedPSP   \Sat{\clF}$. \improve{\todo{Why? Is it true? Yes! Trivially.}}
\item   If also $\clI$ is closed under finite disjoint unions, then
$\Sat{\univ{(\LSuma{a}{\clI}{\clF})}} \RedPSP    \Sat{\clF}$.
\end{enumerate}
\end{corollary}

\subsection{PSPACE-hardness: some corollaries of Ladner's construction}

\improve{
\todef{$\LS{4.1}$, $\Grz.2$}
\todef{$\vL$-satisfiable !!! According \cite{BRV-ML}, we speak about sat in models (consistency); add a remark that
for a Kr-compl logic, $\vL$-sat=$\clF$-sat; or use $\vL$-consistency instead? --- good idea!
\\
$\vL$  is $\PSPACE$-hard.}

}

According to Ladner's theorem, every unimodal logic contained in $\LS{4}$ is $\PSPACE$-hard  \cite{Ladner77}.\improve{Engl}
In fact, Ladner's proof yields
$\PSPACE$-hardness for a wider class of modal logics (e.g., for logics contained in the   G\"odel-L\"ob logic $\GL$ or in the
 Grzegorczyk logic $\Grz$), see \cite{Spaan1993complexity}.
With minor modifications, Ladner's construction also works for logics contained in $\LS{4.1}$, $\Grz.2$ etc, for the polymodal case, and in particular -- for sums.

To illustrate this, let us briefly discuss the proof.
Consider a quantified Boolean formula
$\eta=Q_1 p_1 \ldots Q_{m} p_{m} \,\theta$, where $Q_1,\ldots, Q_{m}\ \in \{\EE,\AA\}$, and
$\theta$ is a propositional Boolean formula in variables $p_1,\ldots,p_{m}$.  Choose fresh variables $q_0,\ldots,q_m$.
Let $\trLad{\eta}$ be the
 following unimodal formula:\footnote{This variant of reduction is a slight modification of the one used in \cite{BRV-ML}.}
\def\Bi{
(q_i\imp \Di (q_{i+1}\wedge p_{i+1})\con \Di (q_{i+1}\wedge \neg p_{i+1}))
}
\def\Si{
(
(p_i\imp \Box p_ i)\con(\neg p_i\imp \Box \neg p_ i)
)
}
\begin{align*}
q_0\;\con\; & \bigwedge\nolimits_{i<m} \Boxm (q_i\imp \Di q_{i+1})  \;\con\;  \bigwedge\nolimits_{\substack{i \neq j\\ i,j\leq m}}\Boxm (q_i\imp \neg  q_{j})   \;\con \;
\Box^m(q_m\imp \theta) \;\con
\\
& \bigwedge\nolimits_{\{i<m \mid Q_{i+1}=\AA\}} \Box^i  \Bi  \;\con
\\
& \bigwedge\nolimits_{1\leq i\leq m-1}\Box^{i}(q_i\imp \Box^{\leq m} p_i \vee \Box^{\leq m} \neg p_i)
\end{align*}
where $\Box^m =\neg \Di^m\neg$; likewise for $\Box^{\leq m}$.

Let $\frT_\eta=(W_\eta,R_\eta)$ be the {\em quantifier tree} of $\eta$:
$W_\eta=\bigcup_{k\leq m} \{\sigma\in 2^k\mid \sigma(i)=0\text{ if } Q_{i+1}=\EE \} $;
$\sigma_1 R_\eta \sigma_2$ iff $\sigma_1\subset\sigma_2$ and $\dom{\sigma_2}=\dom{\sigma_1}+1$.
We remark that $\frT_\eta$ is antitransitive (that is, $x R y R z \Rightarrow \neg (xR z)$ holds in $\frT_\eta$), and so is irreflexive.
We have:
\begin{eqnarray}
\eta  \text{ is valid } &\quad \Rightarrow \quad& \trLad{\eta} \text{ is  satisfiable in } (W_\eta,R_\eta^*),\label{eq:Ladn1}\\
\trLad{\eta} \text{ is  satisfiable in a Kripke frame} &\quad \Rightarrow \quad & \eta \text{ is valid;} \label{eq:Ladn2}
\end{eqnarray}
see \cite[Section 6.7]{BRV-ML} for the details.

Consider a unimodal logic $\vL\subseteq \LS{4}$.
Then $(W_\eta,R_\eta^*)$ is an $\vL$-frame, hence every satisfiable in this frame formula is $\vL$-consistent.
So we have:
$$
\eta  \text{ is valid } \quad \Rightarrow \quad \trLad{\eta} \text{ is  satisfiable in } (W_\eta,R_\eta^*)
\quad \Rightarrow \quad \trLad{\eta} \text{ is  $\vL$-consistent} \quad \Rightarrow \quad  \eta \text{ is valid},
$$
that is, $\eta$ is valid iff $\trLad{\eta}$ is  $\vL$-consistent. Thus, the $\vL$-consistency problem is $\PSPACE$-hard;
synonymously, the  (provability problem for the) logic $\vL$ is $\PSPACE$-hard.

This proves  Ladner's theorem, in its classical formulation. And, in fact, it proves more.
Let $\GrzB$ be the logic of   the class of all finite transitive reflexive trees with branching $\leq 2$. We
have the following proper inclusions
$$\LS{4}\subsetneq \LS{4.1} \subsetneq  \Grz  \subsetneq  \GrzB,$$
see, e.g., \cite{CZ}, and \cite{gabbay_deJongh_trees_1974} for the latter inclusion.
Observe that  $(W_\eta,R_\eta^*)$ is a finite transitive reflexive tree. It immediately follows from the above reasonings
that every logic contained in $\Grz$ is  $\PSPACE$-hard. Moreover, the branching of $(W_\eta,R_\eta^*)$ is $\leq 2$; thus,
the result holds for logics contained in $\GrzB$:  \comm{History? Details on this logic?}
$$ \vL\subseteq \GrzB \quad \Rightarrow \quad  \vL \textrm{ is $\PSPACE$-hard}.$$

This formulation does not include an important  logic $\GL$: its frames are irreflexive. However, we can reformulate (\ref{eq:Ladn1}) in the following way:
if $R_\eta\subseteq R\subseteq R_\eta^*$, then
\begin{equation}\label{eq:Ladn3}
\eta  \text{ is valid }  \quad \Rightarrow \quad \trLad{\eta} \text{ is  satisfiable in } (W_\eta,R);
\end{equation}
the proof is straightforward and is an immediate analog of the proof of (\ref{eq:Ladn1}) given in \cite{BRV-ML}.

The condition $R_\eta\subseteq R\subseteq R_\eta^*$ can be reformulated as $R^*\,=\,R_\eta^*$, because $\frT_\eta$ is a tree.
\begin{definition}\label{def:thick-uni}
A class $\clF$ of unimodal frames is {\em thick} if for every  finite transitive reflexive tree $\frT=(T,\leq)$ 
whose branching is bounded by 2 there exist a relation $R$ on $T$ and a frame $\frF\in\clF$ such that
$R^*\;=\;\leq$ and $\frF$ is isomorphic to $(T,R)$.
\hide{
A class $\clF$ of unimodal frames is {\em thick} if for every  finite transitive reflexive tree $\frT=(T,\leq)$ with
branching $\leq 2$ there exists a relation $R$ on $T$ such $R^*=\leq$ and $(T,R)$ is isomorphic to a frames $\frF$ in $\clF$.}
\hide{A class $\clF$ of unimodal frames is {\em thick} if for every  finite transitive reflexive tree $\frT=(T,\leq)$ with
branching $\leq 2$
 there exists a frame $(W,R)\in\clF$ whose transitive reflexive closure $(W,R^*)$ is isomorphic to $\frT$.}
\end{definition}
From (\ref{eq:Ladn3}) we obtain that if $\clF$ is thick then
\begin{equation}\label{eq:thick0}
\eta  \text{ is valid }  \quad \Rightarrow \quad \trLad{\eta} \text{ is  satisfiable in } \clF.
\end{equation}
Thus, for a class $\clF$ of unimodal frames and a unimodal logic $\vL$,
from (\ref{eq:thick0}) and (\ref{eq:Ladn2}) we have:
\begin{equation}\label{eq:thick}
\clF \text{ is thick and }\vL\subseteq \Log\clF \quad \Rightarrow \quad  \vL \textrm{ is $\PSPACE$-hard}.
\end{equation}
In particular, this formulation allows to apply Ladner's construction for logics below $\GL$, since $\GL$ is the logic
of a thick class (consisting of  finite irreflexive transitive trees).

The logic $\LS{4.2}$
is not the logic of a thick class: it does not have trees of height $>1$ within its frames. However, trees are contained as subframes in $\LS{4.2}$-frames. And this is also enough for $\PSPACE$-hardness due to the following relativization argument proposed by E. Spaan \cite{Spaan1993complexity}.

\improve{Check the preliminaries for the subframe notion}
Recall that {\em a subframe} of a frame $\frF=(W,(R_a)_{a\in\AlM})$ is the restriction $\frF\restr V=(V,(R_a\cap (V\times V))_{a\in\AlM})$, where $V\neq\emp$.
For a class $\clF$ of frames, let $\SubFrs{\clF}$ be its closure under the subframe operation:
  $\SubFrs{\clF}=\{\frF\restr V\mid \frF\in \clF \text{ and } \emp\neq V\subseteq \dom{\frF}\}$.
The satisfiability in $\SubFrs{\clF}$ can be reduced
to the satisfiability in $\clF$ (see the proof of Theorem 2.2.1 in \cite{Spaan1993complexity}). Namely,
for an $\AlM$-formula $\vf$ and a variable $q$, let $\trSub{\vf}$  be the $\AlM$-formula inductively defined as follows:
$\trSub{\bot}=\bot$, $\trSub{p}=p$ for variables, $\trSub{\psi_1\imp \psi_2}=\trSub{\psi_1}\imp \trSub{\psi_2}$,
and $\trSub{\Di_a \psi}= \Di_a (\trSub{\psi}\con q)$ for $a\in\AlM$. We put
$\trRel{\vf}= q\con \trSub{\vf},$  where $q$ is the first variable not occurring in $\vf$.
If  $\mM=(\frF,\theta)$ is
 a model and  $V=\theta(q)$, then for every $v\in V$ we have:
$\mM\restr V,v\mo \vf \;\tiff \; \mM,v\mo \trRel{\vf}$; the proof is by induction on $\vf$. This immediately yields
\hide{
\begin{proposition}
Suppose that a variable $q$ does not occur in a $\AlM$-formula $\vf$, $\frF$ is an $\AlM$-frame. Then
$$
q\con \trSub{\vf}\in \Sat{\frF} \quad \tiff \quad  \vf\in \CSat{\frF\restr V} \text{ for some non-empty }V\subseteq \dom{\frF}
$$
\end{proposition}
}
\begin{proposition}[\cite{Spaan1993complexity}]\label{prop:sat-thick-sub}
$\vf$ is satisfiable in $\SubFrs{\clF}$   \;iff  \;
 $\trRel{\vf}$ is satisfiable in  $\clF$.
\end{proposition}
Thus, $\Sat{\SubFrs\clF}$ is polynomial time \improve{\todo{not precise}} reducible to $\Sat\clF$.

It immediately follows from (\ref{eq:thick})  and  Proposition \ref{prop:sat-thick-sub} that
if $\SubFrs{\clF}$ is thick, then the logic of $\clF$ is $\PSPACE$-hard.
Moreover, this holds for every $\vL\subseteq \Log\clF$:   a quantified Boolean formula
$\eta$ is valid iff $\trRel{\trLad{\eta}}$ is  $\vL$-consistent.
Indeed, if $\eta$ is valid,
then  $\trLad{\eta}$ is satisfiable in $\SubFrs{\clF}$ by  (\ref{eq:thick0}), so
$\trRel{\trLad{\eta}}$ is satisfiable in $\clF$ by Proposition \ref{prop:sat-thick-sub}, hence $\trRel{\trLad{\eta}}$  is $\vL$-consistent.
The converse implication follows from (\ref{eq:Ladn2}). Thus, we have
\begin{theorem}[A corollary of \cite{Ladner77} and \cite{Spaan1993complexity}]
Let $\clF$ be a class of unimodal frames. If $\SubFrs{\clF}$ is thick, then
every unimodal $\vL\subseteq \Log\clF$ is $\PSPACE$-hard.
\end{theorem}

For an $\AlM$-frame $\frF=(W,(R_a)_{a\in\AlM})$ and $a\in \AlA$, let $\reduct{\frF}{\{a\}}$ be its reduct $(W,R_a)$.
A class $\clF$ of $\AlM$-frames is said to be {\em thick} if for some $a\in \AlM$ the class $\{\reduct{\frF}{\{a\}}\mid \frF\in\clF\}$ is thick.
\begin{proposition}
Let $\clF$ and $\clI$ be classes of $\AlM$-frames. If $\clF$ is non-empty and
$\SubFrs{\clI}$ is thick, then $\SubFrs{\LSum{\clI}{\clF}}$ is thick.
\end{proposition}
\begin{proof}
Follows from Definitions \ref{def:sum-poly-extra} and \ref{def:thick-uni}.
\improve{DC; give details}
\end{proof}

\begin{corollary}\label{cor:Hard}
Let $\clF$ and $\clI$ be classes of $\AlM$-frames. If $\clF$ is non-empty and
$\SubFrs{\clI}$  is thick, then $\Sat{\LSum{\clI}{\clF}}$ is $\PSPACE$-hard.
\end{corollary}
\improve{ (in particular, if $\clI$ contains all finite trees)}

\subsection{Examples}\label{sub:examples}
For many logics, Theorem \ref{thm:main_compl_res}
gives a uniform proof of decidability in $\PSPACE$ ($\PSPACE$-completeness in view of Corollary \ref{cor:Hard}).
Let us illustrate it with certain examples.

\begin{example}\label{ex:S4andAround}
It is well-known that the logic $\LS{4}$ as well as its expansion with the universal modality are $\PSPACE$-complete  \cite{Ladner77},\cite{SpaanUniv}.

Let us prove it via sums.   Recall that clusters are frames of the form $(C,C\times C)$.
Every preorder $\frF$ is isomorphic to the sum $\LSum{C\in \clR{\frF}}(C,C\times C)$ of its clusters over its skeleton $\clR{\frF}$. The logic $\LS{4}$ has the finite model property,
so it is enough to consider only finite indices, and
hence $\LS{4}$ is the logic of the class
\begin{equation}\label{eq:s4viaSums}
\LSum{\text{finite partial orders}}{\mathrm{clusters}}
\end{equation}
Thus, we have:
$$\Sat{(\mathrm{preorders})}  \; \RedPSP  \; \Sat{(\mathrm{clusters})}$$
or dually
$$\LS{4}  \; \RedPSP  \; \LS{5}.$$

The satisfiability problem in clusters is in $\NP$ (one can easily check that
a formula $\vf$ is satisfiable in a cluster iff it is satisfiable in a cluster of size $\lvf$;
see, e.g., \cite[Section 18.3]{CZ}). 
Now $\PSPACE$-completeness of $\LS{4}$ follows from (\ref{eq:s4viaSums}): we have $\PSPACE$-upper bound by Theorem \ref{thm:main_compl_res}(\ref{thmItem:main_compl_res-1}); $\PSPACE$-hardness is given by Corollary \ref{cor:Hard}.
Moreover, Theorem \ref{thm:main_compl_res}(\ref{thmItem:main_compl_res-2}) gives $\PSPACE$-upper bound of the logic  \mbox{$\univL{\LS{4}}=\Log\{(W,W\times W,R)\mid (W,R) \text{ is a preorder}\}$}.

Changing the class of summands, we  obtain $\PSPACE$-completeness for other logics.
Let $\frS_0$ and $\frS_1$ be an irreflexive and a reflexive singleton, respectively.

If we add $\frS_0$ to the class of summands-clusters, then we obtain $\PSPACE$-completeness of $\LK{4}$ and $\univL{\LK{4}}$.

Letting the class of summands be $\{\frS_0\}$, we obtain $\PSPACE$-completeness for the logics $\GL$. If
the class of summands consists of a reflexive singleton $\frS_1$, the above reasonings give $\PSPACE$-completeness of the
Grzegorczyk logic, and if the class of summands is $\{\frS_0, \frS_1\}$  --- for the weak Grzegorczyk logic (recall that
the latter logic is characterized by frames whose reflexive closures are non-strict Noetherian orders \cite{LitakGrz2007}; this
class can be represented as $\LSum{\NPO}\{\frS_0, \frS_1\}$).
\end{example}
The results discussed in Example \ref{ex:S4andAround} are well-known \cite{Ladner77},\cite{Spaan1993complexity},\cite{SpaanUniv}. \improve{Ref on GL?}
To the best of our knowledge, the following result has never been published before.
\begin{example}\label{ex:wk4compl}
The {\em weak transitivity logic} $\wK4$ is the logic of frames satisfying the condition $x R z R  y \Rightarrow xR y\vee x=y.$
The {\em difference logic} $\DL$ is the logic of the frames such that $xR y$ whenever $x\neq y$; let $\clF_{\neq}$ be the class of such frames.
The logic $\wK4$  has the finite model property \cite{esakia2001weak}, \improve{DC the ref}
hence it is the logic of the class
\begin{equation}\label{eq:reprwK4}
\LSum{\text{finite partial orders}}\clF_{\neq},
\end{equation}
see Example \ref{ex:wk4semant}. It is not difficult to check that $\Sat{\clF_{\neq}}$ is in $\NP$  (like in the case of clusters, the size of a countermodel is linear in the length of
a formula), and by Theorem \ref{thm:main_compl_res} we obtain that \improve{ref?}
$$
\wK4\RedPSP \DL\in \PSPACE,
$$
so $\wK4$ is in $\PSPACE$ (and $\PSPACE$-complete: $\PSPACE$-hardness follows from \cite{Ladner77},
since $\wK4\subseteq \LS{4}$, or from the representation (\ref{eq:reprwK4}) and Corollary \ref{cor:Hard}). Moreover, since the class $\clF_{\neq}$ is preconical, $\univL{\wK4}$ is in $\PSPACE$.
\end{example}
\begin{example}
The logic $\wK4.2$ is the logic of weakly transitive frames (considered in the above example) satisfying the Church-Rosser property
$xRy_1\,\&\,xRy_2\;\Rightarrow \;\EE z\, (y_1Rz\,\&\,y_2Rz)$.
Observe that every frame validating $\wK4.2$ is either in $\clF_{\neq}$, or is isomorphic to the sum of a frame validating $\wK4$ and a frame in $\clF_{\neq}$.
$$
\Frames{(\wK4.2)} \quad = \quad \clF_{\neq} \cup (\Frames{(\wK4)}\;\plusV{0}\;\clF_{\neq}).
$$
Now that $\wK4.2$ is decidable in $\PSPACE$ follows from the previous example and the following theorem.
\end{example}


\begin{theorem}
Let $\clF$ and $\clG$ be classes of $\AlM$-frames, and $a\in \AlM$.
If $\Sat{\univ{\clG}}\RedPSP\Sat{\univ{\clF}}$ (in particular, if $\Sat{\univ{\clG}}$ is in $\PSPACE$), then
$\Sat{\univ{(\clF\plusa\clG)}}\RedPSP\Sat{\univ{\clF}}$, and consequently
$\Sat{(\clF\plusa\clG)}\RedPSP\Sat{\univ{\clF}}$.
\improve{First sat - weak; DC!!!!}
\end{theorem}
\begin{proof}
By Propositions \ref{prop:CsatToSatUniv} and \ref{prop:satU_to_CSat},  from
$\Sat{\univ{\clG}}\RedPSP\Sat{\univ{\clF}}$ we obtain $\CSat{\clG}\RedPSP\CSat{\clF}$.
Using  Lemma \ref{psp:lem:satsumm-decomp-top-poly},
from the latter reduction we obtain
$\CSat{(\clF\plusa\clG)}\RedPSP\CSat{\clF}$.
Using Propositions \ref{prop:CsatToSatUniv} and \ref{prop:satU_to_CSat} again, we obtain
that $\Sat{\univ{(\clF\plusa\clG)}}\RedPSP\Sat{\univ{\clF}}$.
\end{proof}
\begin{remark}
This theorem can be strengthened in two aspects. First, it can be formulated for reductions stronger than $\RedPSP$.
Another observation is that instead of the class $\clF\plusa\clG$, i.e., a class of sums over $\frI=(2,<)$,
one can consider sums over an arbitrary finite indexing frame $\frI$; in this case, a reduction (with several oracles)
from sums to summands would follow from Lemma \ref{psp:lem:extrernal}.
\end{remark}

We conclude this section with the following example.
In topological semantics ($\Di$ is for the derivation),
$\wK4$ is the logic of all topological spaces \cite{esakia2001weak}. \improve{Shehtman}
In \cite{GuramEsakiaDavid2009} it was shown that the logic $\T0$
of all $T_0$-spaces is equal to the logic
of all finite weakly transitive frames where clusters contain at most one irreflexive point.
The satisfiability for such clusters is in $\NP$, and we obtain
\begin{corollary}
$\T0$ is $\PSPACE$-complete.
\end{corollary}

\section{Variations}\label{sec:variations}

In this section we are interested in polymodal logics
which can be characterized by models obtained via  \improve{engl}
multiple applications of the sum operation. An important example of such logic is {\em Japaridze's polymodal logic} \cite{JaparJapar,Bekl-Jap}; we also consider {\em lexicographic products of modal logics} introduced in \cite{Balb2009}
and {\em refinements of logics} introduced in \cite{BabRyb2010}; see Sections \ref{subs:Japar} and \ref{sub:refandlexProduct} below.
We anticipate them with some general observations.

\subsection{Iterated sums over unimodal indexes}

Let $\clF$ be a class of $\Al$-frames and  $\clI$  a class of 1-frames.
Let
$\vct{a}=(a_0,\ldots, a_{s-1})$ be a finite sequence of elements of $\Al$.
If $\vct{a}$ is the empty sequence, let
{\em $\vct{a}$-iterated sums of $\clF$} be the elements of $\clF$. For
$0<s<\omega$, let
{\em $\vct{a}$-iterated sums of $\clF$} be  $a_0$-sums of $(a_1,\ldots,a_{s-1})$-iterated sums:
that is, an {\em $\vct{a}$-iterated sum of frames in $\clF$ over frames in $\clI$} is a frame of form
$\LSuma{a_0}{\frI}{\frH_i}$, where $\frI\in \clI$ and every $\frH_i$ is an $(a_1,\ldots,a_{s-1})$-iterated sum of frames in $\clF$ over frames in $\clI$.
The class of all such sums is denoted by
$\LSuma{\vct{a}}{\clI}{\clF}$.

Since $\LSuma{\vct{a}}{\clI}{\clF}$ is $\LSuma{a_0}{\clI}{\LSuma{(a_1,\ldots,a_{s-1})}{\clI}{\clF}}$, \improve{Bad notation?}
by Theorem \ref{thm:main_compl_res} we obtain
\begin{corollary}\label{cor:iteratedPSpace}
Let $\AlM<\omega$, $\clF$ a class of $\AlM$-frames,  and $\vct{a}$ a finite sequence of elements of $\Al$.
If
  $\clI$  is a class of Noetherian orders that contains all finite trees and is closed under finite disjoint unions,
  then
$\Sat{\univ{(\LSuma{\vct{a}}{\clI}{\clF})}} \RedPSP    \Sat{\univ{\clF}}$.
\end{corollary}

Theorem \ref{thm:smll-tree-for-Csat} reduces satisfiability on sums over Noetherian orders to sums over finite trees.
This theorem can also be extended for the case of iterated sums \cite{AiML2018-sums}. Namely, let
$\vct{a}=(a_0,\ldots a_{s-1})\in \Al^s$, $0<s<\omega$.
Then
for every $\Al$-tie $\tau=(\vf,\Phi,\oGamma)$ we have:
\begin{equation}\label{eq:polymod-smll-tree-for-Csat}
\tau \text{ is satisfiable in }  \LSuma{\vct{a}}{\NPO}{\clF} \tiff \tau \text{ is satisfiable in }  {\Disk{\lvf}\LSuma{\vct{a}}{\Tr(\lvf,\lvf)}{\clF}}.
\end{equation}
This can be proven by induction of the length of $\vct{a}$  with the help of the following lemma (see the proof of \cite[Theorem 5.2]{AiML2018-sums} for the details):

\begin{lemma}\label{psp:lem:sums-disums-e34}
Let
$\clF$ be a class of $\AlM$-frames and $a\in\AlM$. Then
every frame in $\LSuma{a}{\NPO}{\bigsqcup{\clF}}$
is isomorphic to a frame in $\LSuma{a}{\NPO}{\clF}$.
\end{lemma}
\begin{proof}
By Theorem \ref{thm:basicTrPr}(\ref{item3c:basicTrPr}),
a sum of form
$\LSuma{a}{i\in\frI}{{\bigsqcup}_{j\in J_i}\frF_{ij}}$ is
isomorphic to the sum
$\LSuma{a}{(i,j)\in \LSum{k\in \frI}{(J_k,\emp)}}{\frF_{ij}}$.
It remains to observe that if $\frI=(I,<)\in \NPO$ and  $(J_i)_{i\in I}$ is a family of non-empty sets, then
$\LSum{\frI}{(J_i,\emp)}\in \NPO$. \improve{Improve!}
\end{proof}
In view of Proposition \ref{prop:sat_via_cSat}, from  (\ref{eq:polymod-smll-tree-for-Csat}) we obtain
\begin{corollary}\label{psp:thm:itersums-trees}
Let
$\clF$ be a class of $\Al$-frames, $s<\omega$, and $\vct{a}=(a_0,\ldots a_{s-1})\in \Al^s$.
Then for every $\Al$-formula $\vf$ we have:
$$\vf \text{ is satisfiable in }  \LSuma{\vct{a}}{\NPO}{\clF} \tiff \vf \text{ is satisfiable in }  \LSuma{\vct{a}}{\Tr(\lvf,\lvf)}{\clF}.$$
\end{corollary}

\improve{
\todo{Explain that this corollary is a more subtle instrument than the previous one (about PSPACE-red on iterated sums).
This allows to describe the analog of the algorithm \ref{alg:SatSumFa-poly}. Announce that is will be used to prove complexity of Japar.}

\todo{-  Iterated versions of fmp and dec \\
- No, let us formulate it for lex only in this paper? \\
- No, we can't: we need to apply theorem \ref{Thm:CorrectAlgTree} for Japar; To it is DONE.}

}

\medskip
\begin{remark}
The operation of iterated sum can result in very tangled structures.
In this remark, we expand the formal definition of this operation.


Let $0\leq s<\omega$. Consider a tree $(T,<)$ such that every maximal chain in $T$ has $s+1$ elements.
Let $I$ be the set of maximal elements of $(T,<)$, $J=T\setminus I$,
and for $i\in J$, let $\scc{i}$ be the set of immediate successors of $i$.
A structure $\frT=(T,<,(S_i)_{i\in J})$ such that $S_i$ is a binary relation on $\scc{i}$ is called an {\em indexing tree}.
If also
$(\scc{i},S_i)\in \clI$ for every $i\in J$, then $\frT$ is an {\em indexing tree for $\clI$}.
For 
a sequence $\vct{a}=(a_0,\ldots, a_{s-1})\in \Al^s$,
and a family $(\frF_i)_{i\in I}$ of $\Al$-frames, we define
{\em the $\vct{a}$-sum of $(\frF_i)_{i\in I}$ over the indexing tree $\frT$}, in symbols $\LSuma{\vct{a}}{\frT}{\frF_i}$,
as the following $\Al$-frame $(W,(R_a)_{a\in \Al})$.
The domain of this structure is the disjoint union of the domains of $\frF_i$:
$$W=\{(i,w)\mid i \text{ is maximal in } \frT \text{ and } w \text{ is in } \frF_i
\}
$$
To define the relations $R_a$, consider $(i,w),(j,v)\in W$. If $i=j$, then, as usual, we
put  $(i,w)R_a(j,v)$ iff $w R_{i,a} v$, where $R_{i,a}$ denotes the $a$-th relation in $\frF_i$.
Assume that $i\neq j$. Let $k$ be the infimum $\inf{\{i,j\}}$, and let $h$ be the height of $k$ in $\frT$, that is the number of elements in $\{l\in T\mid l<k\}$.
There are unique immediate successors $i'$ and $j'$ of $k$ such that
$i'\leq i$ and $j'\leq j$.
Hence $0\leq h <s$, $i'\neq j'$, and $i',j' \in \scc{k}$.
We put $(i,w)R_a(j,v)$  iff
$i'S_kj'$
and
$a$ is the $h$-th element of $\vct{a}$.

The class $\LSumAll{\clI}{\clF}$ of {\em iterated
sums of frames in $\clF$ over frames in $\clI$} consists of such structures where $\frT$ is an indexing tree for $\clI$ and
all $\frF_i$ are in $\clF$. One can see that for a fixed $\vct{a}$, the elements of $\LSumAll{\clI}{\clF}$
are (up to isomorphisms) $\vct{a}$-iterated sums $\LSuma{\vct{a}}{\clI}{\clF}$.
\end{remark}

\subsection{Lexicographic sums}
The sum operation does not change the signature.
In many cases (see below) it is convenient to characterize a polymodal logic via the following modification of $a$-sums.
\begin{definition}
Let $\frI=(I,S)$ be a unimodal frame, $(\frF_i)_{i\in I}$ a family of $\Al$-frames,
and $\frF_i=(W_i,(R_{i,a})_{a\in \Al})$.
The {\em lexicographic sum} $\LSuml{\frI}{\frF_i}$ is the $(1+\Al)$-frame
$\left(\bigsqcup_{i \in I} W_i, S^{\lex}, (R_a)_{a<N}\right)$,  where
\begin{eqnarray*}
(i,w) S^\lex  (j,u) &\quad\tiff\quad&   i S j, \\
(i,w) R_a  (j,u) &\quad\tiff\quad&   i = j \;\&\;w R_{i,a}  u.
\end{eqnarray*}
For a class  $\clF$ of $\Al$-frames and a class  $\clI$ of 1-frames, we define $\LSuml{\clI}{\clF}$
as the class of all sums
$\LSuml{\frI}{\frF_i}$, where $\frI \in \clI$ and all $\frF_i$ are in $\clF$.
\end{definition}

Notice that lexicographic sums depend on reflexivity of the indexing frame (unlike sums considered in previous sections, see Remark \ref{rem:independOfRefl}.) This difference is explained in Propositions \ref{prop:lex-asums1} and \ref{prop:lex-asums2} below.

Recall that for an $\Al$-frame $\frF=(W,(R_i)_{i\in \Al})$,
$\univ{\frF}$ is the $(1+\Al)$-frame $(W,W\times W, (R_i)_{i\in \Al})$.
Let
$\empL{\frF}$ be the $(1+\Al)$-frame  $(W,\emp, (R_a)_{a<\Al})$; for a class $\clF$ of frames, let
$\empL{\clF}=\{\empL{\frF}\mid \frF\in\clF\}$.

The following is immediate from the definitions: \comm{DChecked}
\begin{proposition}\label{prop:lex-asums1}
For a unimodal frame $\frI=(I,S)$ and a family  $(\frF_i)_{i\in I}$  of $\Al$-frames,
$$
\LSuml{\frI}{\frF_i}\; =
\;\LSum{\frI'}{\frF'_i}\;=\;\LSuma{0}{\frI}{\frF'_i},
$$
where
$\frI'$ is the $(1+\Al)$-frame $(I,S,\empc)$, and for $i\in I$,
$\frF_i'=\univ{\frF_i}$ whenever
$i$ is reflexive in $\frI$, and $\frF_i'=\empL{\frF_i}$ otherwise.
\end{proposition}

\begin{proposition}\label{prop:lex-asums2}
Let $\clF$  be a class of $\Al$-frames.
\begin{enumerate}\label{prop:lex-asums2:item1}
\item If $\clI$ is a class of irreflexive 1-frames, then
$\LSuml{\clI}{\clF}=\LSuma{0}{\clI}{\empL{\clF}}$.
\item If $\clI$ is a class of reflexive 1-frames, then   $\LSuml{\clI}{\clF}=\LSuma{0}{\clI}{\univ{\clF}}$.
\end{enumerate}
\end{proposition}
\begin{proof}
Immediate from Proposition \ref{prop:lex-asums1}.
\end{proof}

\improve{
\todo{

1. Comment how we translate results for sums to lex. sums.
Certain formulations (or in 3?). Corollary for Balbiani? Ryb-Bab here?}

\todo{
2. Motivation of iterated sums.
}
}

Given a class $\clF$ of frames in an alphabet $\AlA$ \improve{\todo{Do we use this wording}}
and a class of unimodal frames $\clI$, let
{\em $0$-iterated lexicographic sums} be elements of $\clF$,
and for $n<\omega$, let
{\em $(n+1)$-iterated lexicographic sums} be  lexicographic sums of $n$-iterated sums.
The class of $n$-iterated lexicographic sums of frames in $\clF$ over frames in $\clI$ is denoted by $\LSumIt{n}{\clI}{\clF}$.

The next fact is the iterated version of the above proposition.
\improve{$\emp$-shift; $R$-shift; or: universal $\AlB$-lift; idle $\AlB$-lift; }
For a class $\clF$ of $\Al$-frames and $n<\omega$, let
$\empLit{\clF}{n}$ be the class of $(n+\Al)$-frames $(W,(\emp)_n,(R_a)_{a\in \Al})$
such that
$(W,(R_a)_{a\in \Al})\in \clF$. Likewise, let
$\univLit{\clF}{n}$ be the class of frames
$(W,(S_a)_{a\in n+\Al})$ such that $S_a=W\times W$ for $a<n$,
and $(W,(S_{n+a})_{a\in \Al})\in\clF$.
\begin{proposition}\label{prop:lex-asums3}
Let $\Al\leq \omega$, $0<n<\omega$, and let
$\clF$ be a class of $\Al$-frames.
\begin{enumerate}
\item \label{prop:item1-lex-asums3}
If $\clI$ is a class of irreflexive 1-frames, then
$\LSumIt{n}{\clI}{\clF}=\LSuma{(0,\ldots,n-1)}{\clI}{\empLit{\clF}{n}}$.
\item
If $\clI$ is class of reflexive 1-frames, then
$\LSumIt{n}{\clI}{\clF}=\LSuma{(0,\ldots,n-1)}{\clI}{\univLit{\clF}{n}}.
$
\end{enumerate}
\end{proposition}
\begin{proof}
Follows from Propositon \ref{prop:lex-asums1} by a straightforward induction on $n$.
\end{proof}
\improve{
Notation: 
$\LSuml{\clI_0;\ldots;\clI_{\AlB-1}}{\clF})$

$\LSuma{0;\ldots;\AlB-1}{\clI_0;\ldots;\clI_{\AlB-1}}{\clF}$

Perhaps, it will complicate the reading...
}
\improve{
Expanded version; do we need it?
\begin{proof}
We only consider the irreflexive case, the
proof for the reflexive case is entirely analogous.

First, observe that $\empL{(\LSuma{a}{\clI}{\clF})}=
\LSuma{a+1}{\clI}{\empL{\clF}}$ for any $a\in \Al$. Hence,
\begin{equation}\label{eq:lex-asums3}
\empL{(
\LSuma{0}{\clI_0}{\ldots{\LSuma{\AlB-1}{\clI_{\AlB-1}}{
\clF
}}}
)}
=
\LSuma{1}{\clI_0}{\ldots{\LSuma{\AlB}{\clI_{\AlB-1}}{
\empL{\clF}
}}}
\end{equation}
(an easy induction on $\AlB$.)

Let $\clF_0=\clF$ and $\clF_{n+1}=\empL{\clF_n}$. By induction on $\AlB$, we show that
$\LSuml{\clI_0}{\ldots{\LSuml{\clI_{\AlB-1}}{
\clF
}}}$ is
$
\LSuma{0}{\clI_0}{\ldots{\LSuma{\AlB-1}{\clI_{\AlB-1}}{
\clF_\AlB
}}}.$
We have \improve{$\AlB=1$?}
\begin{multline}
\LSuml{\clI_0}{\ldots{\LSuml{\clI_{\AlB-1}}{
\clF
}}}
=
\LSuml{\clI_0}{
\LSuma{0}{\clI_1}{\ldots{\LSuma{\AlB-2}{\clI_{\AlB-1}}{
\clF_{\AlB-1}
}}}
}
=
\\
=
\LSuma{0}{\clI_0}{
\empL{(
\LSuma{0}{\clI_1}{\ldots{\LSuma{\AlB-2}{\clI_{\AlB-1}}{
\clF_{\AlB-1}
}}}
)}
}
=
\LSuma{0}{\clI_0}{\ldots{\LSuma{\AlB-1}{\clI_{\AlB-1}}{
\clF_\AlB
}}};
\end{multline}
the first equality holds by induction hypothesis; the second follows from Proposition \ref{prop:lex-asums2};
the third equality follows from  (\ref{eq:lex-asums3}).
\end{proof}
}

\improve{General corollaries on complexity? Words on the general machinery we described earlier?}
\improve{

\todo{
4. FMP-result for iterated sums, complexity.
}
}
Propositions \ref{prop:lex-asums2}  and \ref{prop:lex-asums3} allow to expand our results on the finite model property and complexity to
the logics of (iterated) lexicographic sums.

\subsection{Japaridze's polymodal logic}\label{subs:Japar}
In this section we show that Japaridze's polymodal provability logic $\GLP$ is decidable in $\PSPACE$.
First, this theorem was proven in \cite{ShapJapar}. Here we
provide a version of the proof based on Theorem \ref{Thm:CorrectAlgTree}.

\improve{Perhaps, some steps (finite signature conservativity) can be simplified}

$\GLP$ is a normal modal logic in the language $\ML(\omega)$. This system was introduced in \cite{JaparJapar} and plays an important
role in proof theory (see, e.g., \cite{Bekl2004Ordinal}).
 $\GLP$ is known to be Kripke incomplete, so we cannot directly apply our tools to analyze it.
\improve{More history: decidability - Japar? Incompleteness- Lev?}
However,
in \cite{Bekl-Jap}, L. Beklemishev introduced a modal logic $\J$, a Kripke complete approximation of $\GLP$.
Semantically, $\J$ is characterised as the logic of frames called {\em stratified},
or {\em hereditary partial orderings} (\cite[Section 3]{Bekl-Jap}).
They are defined as follows.

\begin{definition}
For $\AlA\leq \omega$, let
$\frS_\Al$ be $(\{0\},(\emp)_\Al)$, a singleton $\Al$-frame with empty relations.
For $\Al<\omega$, let $\clJ(\Al)$ be the class
of $\Al$-iterated lexicographic sums
$$\LSumIt{\Al}{\NPO}\{\frS_\omega\}=\ILSum{\Al}{\NPO}{\{\frS_\omega\}}.$$
The class $\clJ$
of {\em hereditary partial orderings} is the class $\bigcup_{\Al<\omega}\clJ(\Al)$.
The {\em logic $\J$} is defined as the logic of the class $\clJ$.
\improve{Do we define a class-formula by induction?
(No problem? Similarly to the transitive closure?...)

If this is indeed a problem, we can remove the definition of $\clJ$ --- we do not need it}

\hide{
In particular, $\clJ(0)=\{\frS_0\}$, and $\clJ(1)$ (up to isomorphisms) is $\NPO$.
The class $\clJ$
of {\em hereditary partial orderings}
is the class of $\omega$-frames $\frF=(W,(R_a)_{a<\omega})$
such that for some $\Al<\omega$, the reduct $(W,(R_a)_{a<\Al})$ of $\frF$ is in $\clJ(\Al)$, and $R_a=\emp$ for $a\geq \Al$. }
\end{definition}
In \cite{Bekl-Jap}, it was shown that
$\GLP$ is
polynomial time many-to-one reducible to
$\J$:
there exists a polynomial time computable  $f:\ML(\omega)\to \ML(\omega)$ such that
$\vf\in \GLP \; \tiff \; f(\vf)\in \J$
(for an explicit description of $f$, see \cite[Lemma 3.4]{BeklFernJoost-Lin}).

\medskip

Our aim is to show that $\J$ is in $\PSPACE$. For our purposes, it is more convenient to work with $a$-sums.
Since ${\frS_\omega}^{(\emp)_\Al}=\frS_\omega$, from Proposition  \ref{prop:lex-asums3}(\ref{prop:item1-lex-asums3}) we have:
\begin{proposition}\label{prop:JhatItem1} For every $\Al<\omega$,
$\clJ(\Al)=\LSuma{0}{\NPO}{\ldots\LSuma{\Al-1}{\NPO}{\{\frS_\omega\}}}$.
\end{proposition}
Formally, for every $\Al<\omega$, each frame in $\clJ(\Al)$ has infinitely many relations.
Next facts allow us to consider frames of finite signatures.
Let $$\hat{\clJ}(\Al)= \LSuma{0}{\NPO}{\ldots\LSuma{\Al-1}{\NPO}{\{\frS_\Al\}}}.$$
For an alphabet $\AlC$, a $\AlC$-frame $\frF=(W,(R_a)_{a\in\AlC})$, and $\AlB\leq\AlC$,
let $\frF^{\restr \AlB}$ be the reduct $(W,(R_a)_{a\in\AlB})$;
for a class $\clF$ of $\AlC$-frames, $\clF^{\restr \AlB}=\{\frF^{\restr\AlB} \mid \frF\in\clF\}$.

\begin{proposition}\label{prop:Jhat} For every $\Al<\omega$, we have:
\begin{enumerate}
\item  \label{item:Jhat2}
$\clJ(\Al)^{\restr \Al}=\hat{\clJ}(\Al)$;
\item \label{item:Jhat3}
If $\Al>0$, then every frame in $\clJ^{\restr \Al}$ is isomorphic to a frame in $\hat{\clJ}(\Al)$.
\end{enumerate}
\end{proposition}
\begin{proof}
From Definition \ref{def:sum-poly-extra} it is immediate that
for alphabets $\AlB\leq \AlC$ and classes of $\AlC$-frames
$\clF,\clI$,
\begin{equation}\label{eq:red}
\left(\LSum{\clI}{\ldots
\LSum{\clI}{\clF}
}\right)^{\restr \AlB}
=
\LSum{\clI^{\restr \AlB}}{\ldots
\LSum{\clI^{\restr \AlB}}{\left(\clF^{\restr \AlB}\right)}
}.
\end{equation}
\hide{
\begin{equation}\label{eq:red}
(\LSum{\clI}{\ldots
\LSum{\clI}{\clF})
}^{\restr \AlB}
=
\LSum{\clI^{\restr \AlB}}{\ldots
\LSum{\clI^{\restr \AlB}}{\clF^{\restr \AlB}}
}.
\end{equation}
}
\hide{
\begin{equation}\label{eq:red}
(\LSum{\clI_0}{\ldots
\LSum{\clI_s}{\clF})
}^{\restr \AlB}
=
\LSum{\clI_0^{\restr \AlB}}{\ldots
\LSum{\clI_s^{\restr \AlB}}{\clF^{\restr \AlB}}
}.
\end{equation}
}
\improve{Reread; seems to be to general enough to obscure}

Clearly, ${\frS_\omega}^{\restr \Al}$ is $\frS_\Al$. Now we obtain (1) from Proposition \ref{prop:JhatItem1}  and Definition \ref{def:asum}.

To prove (2), let us fix $\AlB<\omega$ and show that every frame in $\clJ(\AlB)^{\restr \Al}$ is isomorphic to a frame in $\hat{\clJ}(\Al)$.
\improve{It is unclear what is going on here}

First,
observe that $\clJ(\AlB+1)$ contains copies of all frames in $\clJ(\AlB)$.
Indeed,
every frame in the class
$
\LSuma{0}{\NPO}
{\ldots{\LSuma{\AlB-1}{\NPO}
{
\{\frS_\omega\}
}}}
$
is isomorphic to a frame in $\LSuma{0}{\NPO}
{\ldots{\LSuma{\AlB-1}{\NPO}
{
\LSuma{\AlB}{\NPO}
{\{\frS_\omega\}}
}
}}$, since
$\LSuma{\AlB}{\NPO}
{\{\frS_\omega\}}$
contains a copy of $\frS_\omega$.
Hence, if $\AlB< \AlA<\omega$, $\clJ(\AlA)$ contains copies of all frames in $\clJ(\AlB)$.

The case $\AlA< \AlB$ is more interesting.
Consider the class $\clG=\LSuma{\Al-1}{\NPO}{\ldots{\LSuma{\AlB-1}{\NPO}\{\frS_\omega\}}}$.
By (\ref{eq:red}) and Proposition \ref{prop:JhatItem1}, we have:
$$
\clJ(\AlB)^{\restr \Al}=
\left(\LSuma{0}{\NPO}
{\ldots{\LSuma{\AlB-1}{\NPO}
{
\{\frS_\omega\}
}}}\right)^{\restr \Al}
=
\LSuma{0}{\NPO}
{\ldots{\LSuma{\AlA-2}{\NPO}
{
\clG^{\restr \Al}
}}}.
$$
By (\ref{eq:red}) again, $\clG^{\restr \Al}$
can be represented as sums of disjoin unions of singletons:
$$
\left(\LSuma{\Al-1}{\NPO}{\ldots{\LSuma{\AlB-1}{\NPO}\{\frS_\omega\}}}\right)^{\restr \Al}=
\LSuma{\Al-1}{\NPO}{\,\underbrace{\bigsqcup\ldots \bigsqcup}_{\AlB-\AlA\text{ times}}\{\frS_\Al\}},
$$
since a sum over the structure with all relations empty is just a disjoint sum.
By Lemma \ref{psp:lem:sums-disums-e34},
 every frame in $\clG^{\restr \Al}$ is isomorphic to a frame in the class
 $\LSuma{\Al-1}{\NPO}{\{\frS_\Al\}}$. It follows that $\hat{\clJ}(\Al)$ contains copies of
all frames in $\clJ(\AlB)^{\restr \AlA}$.
\end{proof}
\improve{\todo{NUMERATION - equations and items in proposition!} ?}
\begin{proposition}\label{prop:AfragmentOfJ}
For every $\Al<\omega$, $\J\cap\ML(\Al)=\Log(\hat{\clJ}(\Al))$.
\end{proposition}
\improve{
\todo{$\J$ - fix typeset} }
\begin{proof}
By Proposition \ref{prop:Jhat}(\ref{item:Jhat2}),
$\hat{\clJ}(\Al) \subseteq \clJ^{\restr \Al}$, so if $\vf$ is satisfiable in
$\hat{\clJ}(\Al)$, then it is satisfiable in $\clJ$. Conversely, if $\vf\in\ML(\AlA)$
is satisfiable in $\clJ$, then it is satisfiable in
$\clJ^{\restr \Al}$, and hence --- in
$\hat{\clJ}(\Al)$ by Proposition \ref{prop:Jhat}(\ref{item:Jhat3}). \improve{(we ignore the trivial situation when $\Al=0$).}
\end{proof}

\todo{Recheck above}

It is trivial that the satisfiability problem on the singleton $\univ{\frS_\Al}$
is in $\PSPACE$ (in fact, it is in $\NP$).
\improve{CSAT in P!}
From Corollary \ref{cor:iteratedPSpace} and Proposition \ref{prop:Jhat}(\ref{item:Jhat2}), we obtain
\begin{corollary}
For every $\Al<\omega$,  $\Sat{\univ{\hat{\clJ}(\AlA)}}\in \PSPACE$.
\end{corollary}
Hence, every fragment of $\J$ with finitely many
modalities is in $\PSPACE$.
The above corollary does not directly imply that $\Sat\clJ$ is in $\PSPACE$.
But the latter fact is a corollary of Theorem \ref{Thm:CorrectAlgTree} and the observations below. \improve{mention
another theorem - about reduction to small trees. Maybe, other warding - ``inspection of the algorithm}

\begin{lemma}\label{lem:jap-treesums1}
Let $\Al<\omega$.
An $\Al$-formula $\vf$ is satisfiable in $\J$ iff
$\vf$ is satisfiable in
$\LSuma{(0,\ldots,\Al-1)}{\Tr(\lvf,\lvf)}{\{\frS_\Al\}}$.
\end{lemma}
\begin{proof}
By Proposition \ref{prop:AfragmentOfJ}, $\vf$ is satisfiable in $\J$  iff it is
satisfiable in $\hat{\clJ}(\Al)$.  By Corollary \ref{psp:thm:itersums-trees},
the latter is equivalent to the satisfiability of $\vf$ in $\LSuma{(0,\ldots,\Al-1)}{\Tr(\lvf,\lvf)}{\{\frS_\Al\}}$.
\end{proof}

Consider a formula $\vf$ in the language $\ML_\omega$.
Let ${a_0}<\dots<{a_{N-1}}$ be the increasing sequence of indices of all occurring in $\vf$ modalities.
Put $\Nmod{\vf}=N$. Let
$\hat{\vf}$ be the
result of replacing $a_b$-modalities by $b$-modalities in $\vf$.
Note that
\begin{equation}\label{eq:conserv-hat}
\text{$\hat{\vf}$ is an $\Nmod{\vf}$-formula, and $\Nmod{\vf}<\lvf={\#\hat{\vf}}$.}
\end{equation}

\begin{lemma}\cite[Lemma 3.5]{BeklFernJoost-Lin}\label{lem:JapConserv}.
$\vf\in \J  \tiff  \hat{\vf}\in \J$.
\end{lemma}

\improve{
\todo{Below some induction on ``classes''. Ugly. Use  $\Tr(h,b)$ for sets of (non-isomorphic) trees? }}
\begin{theorem}
$\J$ is in $\PSPACE$.
\end{theorem}
\begin{proof}
We describe a decision procedure for the conditional satisfiability on
iterated sums of trees.
For positive $h,b$ and $a\leq\Al<\omega$, we define sums $\clS(h,b,a,\Al)$ as follows:
if $a=\Al$, let $\clS(h,b,a,\Al)$ denote $\{\frS_\Al\}$ for all $h,b$;
if $a<\Al$, let $\clS(h,b,a,\Al)$ denote
$\LSuma{a}
{\Tr{(h,b)}}
{
\LSuma{(a+1,\ldots,\Al-1)}{\Tr(\Al,\Al)}
{
\{\frS_\Al\}
}
}
$.
In particular, $\clS(h,b,\Al-1,\Al)$ is $\LSuma{\Al-1}{\Tr(h,b)}{\{\frS_\Al\}}$, and
$\clS(\AlA,\AlA,0,\Al)=\LSuma{(0,\ldots,\Al-1)}{\Tr(\Al,\Al)}{\{\frS_\Al\}}$.
Using Theorem \ref{Thm:CorrectAlgTree}, we describe the procedure  $\CSatJ$ (Algorithm \ref{alg:SatJ}) that decides
whether a given
$\Al$-tie $(\vf,\vect,\vectU)$ is satisfiable  $\clS(h,b,a,\Al)$:
\begin{lemma}
Let $a\leq \Al<\omega$,
 $(\vf,\vect,\vectU)$ an $\Al$-tie, and $0<h,b<\omega$. Then
$$
(\vf,\vect,\vectU) \text{ is satisfiable in }
\clS(h,b,a,\Al) \tiff  \CSatJ(\vf,\vect,\vectU,h,b,a,\Al)
\text{ returns }\true.
$$
\end{lemma}
\begin{proof}
By induction on $\Al-a$. The case $\Al=a$ is trivial: $\clS(h,b,\Al,\Al)$ consists of a single singleton
$\frS_\Al$. If $a<\Al$, then
$\clS(h,b,a,\Al)$ is
$\LSuma{a}
{\Tr{(h,b)}}
{
\clS(\Al,\Al,a+1,\Al)
}
$. By induction hypothesis,  $\CSatJ(\tau,\Al,\Al,a+1,\Al)$ decides whether a tie $\tau$
is satisfiable in
$\clS(\Al,\Al,a+1,\Al)$.
Observe that for $a<A$, the algorithm $\CSatJ$ is an instance of
the algorithm $\CondSatSumFa$ (Algorithm \ref{alg:SatSumFa-poly} in
Section \ref{sec:compl}), where  $\CSatJ(\tau,\Al,\Al,a+1,\Al)$ is the oracle for the summands
$\LSuma{(a+1,\ldots,\Al-1)}{\Tr(\Al,\Al)}
{
\{\frS_\Al\}
}$.  Now the induction step follows from Theorem \ref{Thm:CorrectAlgTree}.
\end{proof}

\begin{lemma}\label{lem:jap-treesums2}
$\vf$ is satisfiable in $\J$ iff
$\hat{\vf}$ is satisfiable in $\clS(\lvf,\lvf,0,\lvf)$.
\end{lemma}
\begin{proof}
  We have from Lemmas \ref{lem:JapConserv} and  \ref{lem:jap-treesums1}:
$\vf$ is satisfiable in $\J$ iff $\hat{\vf}$ is satisfiable $\J$ iff $\hat{\vf}$ is satisfiable
in $\LSuma{(0,\ldots,\Al-1)}{\Tr(\Al,\Al)}{\{\frS_\Al\}}$, where $\AlA$ can be assumed equal to
$\lvf$ by (\ref{eq:conserv-hat}). The latter class is
$\clS(\AlA,\AlA,0,\Al)$.
\end{proof}

Hence, by Proposition \ref{prop:sat_via_cSat-vect}, we obtain
$$
\vf \text{ is satisfiable in } \J  \tiff  \EE \vect\in2^\lvf\, (\vect(0)=1\, \&\,  (\hat{\vf},\vect,{\mathbf 0})\in
\CSat{\clS(\lvf,\lvf,0,\lvf)}),
$$
where ${\mathbf 0}$\improve{is it clear?} represents the empty condition.

Let us estimate the amount of space used by $\CSatJ$
on the input
$(\hat{\vf},\vect,{\mathbf 0}, \lvf,\lvf,0,\lvf)$.
On each recursive call, either
the argument $a$ increases by 1, or $a$ does not change and
the parameter $h$ decreases by 1. Since $1\leq h \leq \lvf$ and $0 \leq a\leq\lvf$, we obtain that
the depth of recursion is bounded by $\lvf\cdot(\lvf+1)$\todo{DC}.
For $n=\lvf$, at each call $\CSatJ$ needs $\BigO(n^2)$ space
to store new variables.
Hence, to check the satisfiability of $\vf$ in $\clJ$ we need $\BigO(n^4)$ space.
This completes the proof of the theorem.
\improve{a few words on sat in singleton?
Algorithm \ref{alg:SatSumFa-poly} uses a ``trivial'' oracle for $\CSat{\frS_\Al}$
}
\begin{algorithm}[t]
\SetAlgoNoLine
\caption{Decision procedure for $\CSat{\clS(h,b,a,\Al)}$}
\label{alg:SatJ}

\SetKwFunction{mCsat}{$\CSatJ$}
\mCsat{$\vf,\vect,\vectU,h,b,a,\Al$}:${boolean}$
\\ \Indp
\KwIn{An $\Al$-tie $(\vf,\vect,\vectU)$; positive integers $h,b$; $a\leq \Al<\omega$.}
\lIf*{$a=\Al$} {\Return $((\vf,\vect,\vectU)$ is satisfiable in $\frS_\Al$)};
\\
\lIf*{$\CSatJ(\vf,\vect,\vectU,\Al,\Al,a+1,\Al)$} {
\Return ${\true}$\;
}

{
\lIf*{$h>1$}
{
\\ \Indp
\lFor*{$k$ such that $1\leq k\leq b$}
{
\lFor*{$\vectu,\vect_0,\dts,\vect_{k-1}\in 2^{\lvf}$ such that $\vect=\vectu+ \sum_{i<k} \vect_{i}$}
{
\\ \Indp
\lIf*{
$\CSatJ(\vf,\vectu, \vectU\plusa  \sum_{i<k} \vect_{i},\Al,\Al,a+1,\Al)$}
{
\\ \Indp
\lIf*{$\bigwedge\limits_{i<k}\CSatJ(\vf,\vect_i,\vectU,h-1,b,a,\Al)$}
{
\Return ${\true}$\;
}
}
}
}
}
}
\Return ${\false}$.
\end{algorithm}
\end{proof}

\begin{corollary}
Japaridze's polymodal provability logic $\GLP$ is decidable in $\PSPACE$.
\end{corollary}

\begin{remark}
It is immediate that $\GLP$ and $\J$ are $\PSPACE$-hard (e.g., it follows from the fact that
the 1-modal fragments of these logics are the logic $\GL$).  \improve{ref}
From \cite{ChagrovRybakovVarHard2003}, it follows that one-variable fragment of $\GLP$ is $\PSPACE$-hard.
In \cite{PakhHard14}, it was shown that even the constant (closed) fragment of $\GLP$ is $\PSPACE$-hard.
\end{remark}

\subsection{Refinements and lexicographic products of modal logics}\label{sub:refandlexProduct}
In this paragraph we discuss algorithmic properties of logics obtained via the {\em refinement} and the {\em lexicographic product} operations.

The operation of {\em refinement of modal logics} was introduced in \cite{BabRyb2010}.
\begin{definition}\label{psp:def:refinm}
Let $\frF=(W,R)$ be a preorder and
$\clR{F}=(\ov{W},\leq)$ its skeleton.
Consider a family $(\frF_C)_{C\in \ov{W}}$ of $\Al$-frames such that
$\dom{\frF_C}=C$ for all $C\in \ov{W}$.
The  {\em refinement} of $\frF$ {\em by}
$(\frF_C)_{C\in \ov{W}}$ is the  $(1+\AlA)$-frame
$(W,R,(R^\rfn_a)_{a\in \Al})$, where
\begin{eqnarray}\label{psp:eq:def:refinm}
&R^\rfn_a \;\subseteq   \;\bigcup\nolimits_{C\in \ov{W}} C\times C &\text{ for all } a\in \Al,\\
&(W,(R^\rfn_a)_{a\in \Al}) \restr C  \;= \; \frF_C &\text{ for all } C\in \ov{W}.
\end{eqnarray}
For a class $\clI$ of preorders and a class $\clG$ of $\AlA$-frames let $\Refin{\clI}{\clF}$ be the class of all refinements of frames from $\clI$ by frames in $\clF$.
For logics $\vL_1\supseteq \LS{4},\vL_2$, we put $\Refin{\vL_1}{\vL_2}= \Log{\Refin{\Frames{\vL_1}}{\Frames{\vL_2}}}$.
\end{definition}

In \cite{BabRyb2010} it was shown that in many cases the refinement operation preserves the finite model property and decidability.
\hide{In particular, it was shown that
\begin{theorem}\cite{BabRyb2010} \label{thm:refS4S4Fmp}
$\Refin{\LS{4}}{\LS{4}}$ is decidable and
$\Refin{\LS{4}}{\LS{4}}=\Log{\Refin{\QOf}{\QOf}}$, where $\QOf$ is the class of all finite non-empty preorders.
\end{theorem}
}

Refinements can be considered as sums according to the following fact:
\begin{proposition}\label{psp:prop:ref-via-sums}
The refinement of $\frF$ by
frames $(\frF_C)_{C\in \clR{\frF}}$ is isomorphic to the sum
$\LSuma{0}{C \in \clR{\frF}}{\univ{\frF_C}}.$
\end{proposition}
\begin{proof}
The isomorphism is given by $w\mapsto (C,w)$, where  $w\in C$.
\end{proof}
In view
of Theorems \ref{thm:interch1} and \ref{thm:smll-tree-for-Csat}, this observation provides another way to prove the finite model property of refinements (see \cite[Section 5.2]{AiML2018-sums} for more details). Moreover, representation of refinements as sums allows to obtain complexity results according to our theorems in Section \ref{sec:compl}.

Let us illustrate this with the logic $\Refin{\LS{4}}{\LS{4}}$.
In \cite{BabRyb2010}, it was proven that
$\Refin{\LS{4}}{\LS{4}}$ is decidable and that
$\Refin{\LS{4}}{\LS{4}}=\Log{\Refin{\QOf}{\QOf}}$, where $\QOf$ is the class of all finite non-empty preorders.
The latter fact in combination with above proposition yields
$$\Refin{\LS{4}}{\LS{4}}=\Log\LSuma{0}{\Trf}{\univ{\QOf}}.$$
Since the satisfiability problem for the class $\univ{\QOf}$ is in $\PSPACE$ (\cite{SpaanUniv}; see also Example \ref{ex:S4andAround}),
from Theorem \ref{thm:main_compl_res} (and Corollary \ref{cor:Hard}) we obtain
\begin{corollary}
$\Refin{\LS{4}}{\LS{4}}$ is $\PSPACE$-complete.
\end{corollary}

\medskip

A related operation is the {\em lexicographic product of modal logics}
introduced in \cite{Balb2009} by Ph. Balbiani.
\begin{definition}
Consider frames $\frI=(I,S)$ and $\frF=(W,(R_a)_{a\in \Al})$.
Their {\em lexicographic product} $\frI\ltms \frF$ is the $(1+\Al)$-frame
$(I\times W, S^\ltms, (R^\ltms_a)_{a\in \Al})$, where
\begin{eqnarray*}
(i,w) S^\ltms  (j,u) &\quad\tiff\quad&   i S j, \\
(i,w) R^\ltms_a  (j,u) &\quad\tiff\quad&   i = j \;\&\;w R_a  u.
\end{eqnarray*}
In other words, $\frI\ltms \frF$ is the lexicographic sum $\LSuml{\frI}{\frF_i}$, where $\frF_i=\frF$ for all $i$ in $\frI$.

For a class $\clI$ of $1$-frames and a class $\clG$ of $\Al$-frames, the class $\clI\ltms\clF$ is the class of all
products $\frI\ltms \frF$ such that $\frI\in \clI$
and $\frF\in \clF$. For logics $\vL_1,\vL_2$, we put $\vL_1\ltms \vL_2= \Log{(\Frames{\vL_1}\ltms \Frames{\vL_2})}$.
\end{definition}

As examples, consider the logics $\LS{4}\ltms \LS{4}$ and $\GL\ltms \LS{4}$ and show that they are in $\PSPACE$ ($\PSPACE$-complete).
In \cite[Theorem 5.13]{AiML2018-sums}, it was shown that $\LS{4}\ltms \LS{4}$ is equal to $\Refin{\LS{4}}{\LS{4}}$.
Also, it was shown that
$$\GL\ltms \LS{4}=\Log{(\Trf\ltms \QOf)}=\Log\LSuma{0}{\Trf}{\LSuma{1}{\Trf}{\clC}}$$
where $\clC$ is the class of finite frames of form $(C,\emp,C\times C)$
(in terms of lexicographic sums,
$\LSuma{0}{\Trf}{\LSuma{1}{\Trf}{\clC}}$ is the class $\LSuml{\Trf}{\QOf}$ \todo{DC!}); see \cite[Theorem 5.14]{AiML2018-sums} for the proof.
Hence, this logic is in $\PSPACE$ by Corollary \ref{cor:iteratedPSpace}.
\begin{corollary}
The lexicographic products  $\LS{4}\ltms \LS{4}$, $\GL\ltms \LS{4}$ are $\PSPACE$-complete.
\end{corollary}
This result contrasts with the undecidability results for {\em modal products} of transitive logics \cite{GabelaiaAtAlUndec2005}.

\section{Conclusion}\label{sec:concl}

In many cases, sum-like operations preserve the finite model property and decidability \cite{BabRyb2010}, \cite{AiML2018-sums}. In this paper
we showed  that transferring results can be obtained for the complexity of the modal satisfiability problems on sums.
In particular, it follows that for many logics $\PSPACE$-completeness is immediate from their semantic characterizations.

Let us indicate some further results and directions.
\begin{itemize}

\item Sums and products with linear indices

In the linear case, {\em modal products} are typically (highly) undecidable \cite{ReynoldsZakh_Linear2001}.
However, the modal satisfiability problem on the lexicographic squares of dense unbounded linear orders is in $\NP$ \cite{BalbianiMik2013}.
This positive result seems to be scalable according to the following observation:
in many cases, $\vf$ is satisfiable in sums over linear (pre)orders iff
$\vf$ is satisfiable in such sums that the length of indices is bounded by $\lvf$. In this situation there is a stronger than $\RedPSP$ reduction
between the sums and the summands.


\item Further results on the finite model property and decidability of sums

In many cases, the finite model property of a modal logic $L$ can be obtained by a filtration method; in this case we say that
$L$ {\em admits filtration}.
It follows from \cite{BabRyb2010} that filtrations of refinements can be reconstructed from filtrations of components.
This result can be extended for a more general setting of lexicographic sums.

Also, the results obtained in \cite{BabRyb2010} in a combination with Theorem \ref{thm:interch1}  suggest the  following conjecture:
in the case of finitely many modalities, if $\Log{\univ{\clF}}$ has the finite model property and $\Log{\clI}$ admits filtration, then
the logic of $\LSum{\clI}{\clF}$ has the finite model property.

Another fact that we announce relates to the property of local finiteness of logics:
if both the logics of indices and summands are locally finite, then the
logic of lexicographic sums (a fortiori, of  lexicographic products)
is locally finite.

\item Axiomatization of sums

For unimodal logics $L_1,L_2$, let $\LSuml{L_1}{L_2}$ be the logic of the class $\LSuml{\Frames{L_1}}{\Frames{L_2}}$.
Consider the following 2-modal formulas:
$$\alpha=\Di_1\Di_0 p\imp \Di_0 p,
\quad \beta=\Di_0\Di_1 p\imp \Di_0 p,
\quad  \gamma=\Di_0 p\imp \Box_1 \Di_0 p
$$
One can see that these formulas are valid in every lexicographic sum $\LSuml{\frI}{\frF_i}$ (and hence, in every product $\frI\ltms\frF$)
of 1-frames. In many cases, these axioms provide a complete axiomatization of $\LSuml{L_1}{L_2}$.
In particular, the logic $\LSuml{\GL}{\GL}=\Log(\LSuml{\NPO}{\NPO})$, the bimodal fragment of the logic $\J$ considered in Section \ref{subs:Japar},
is the logic $\GL\fus\GL+\{\alpha,\beta,\gamma\}$  \cite{Bekl-Jap}
($L_1 \fus L_2$ denotes the {\em fusion} of $L_1$ and $L_2$,
$L+\Psi$ is for the least logic containing $L\cup \Psi$).
Analogous results hold for various lexicographic products \cite{Balb2009}; e.g.,
$\LS{4}\ltms \LS{4}=\LS{4}\fus\LS{4}+\{\alpha,\beta,\gamma\}$.
We announce the following results:
\begin{enumerate}
\item
If  $\vL_{1} \ast \vL_{2}+\{\alpha,\beta,\gamma\}$ is Kripke complete, and
the class $\Frames{\vL_{1}}$ (considered as a class of models in the classical model-theoretic sense) is first-order definable without equality, then
$$\LSuml{\vL_1}{\vL_2}=\vL_{1} \ast \vL_{2}+\{\alpha,\beta,\gamma\}.$$
\item
Assume that $\Frames{\vL_{2}}$ is closed under direct products and validates the property $\AA x \EE y\, xRy$,
and  that
the class $\Frames{\vL_{1}}$ is first-order definable without equality. If the logic $\vL_{1} \ast \vL_{2}+\{\alpha,\beta,\gamma\}$ is Kripke complete, then
$$\vL_1\ltms\vL_2=\vL_{1} \ast \vL_{2}+\{\alpha,\beta,\gamma\}.$$
\end{enumerate}

At the same time, no general axiomatization results are known for the logics of sums in the sense of Definition \ref{def:sum-poly-extra}.

\item Sum-based operations in the Kripke-incomplete case

The operations we considered so far lead to Kripke complete logics. What could be definition of sums for modal algebras (models, general Kripke frames)?
E.g., can we give a semantic characterization of an important Kripke-incomplete logic $\GLP$ by sums of form
$\LSuma{}{\NPO}\ldots{\LSuma{}{\NPO}{\clC}}$?
\end{itemize}

\section{Acknowledgements}
I am sincerely grateful to the reviewers for multiple important suggestions on the earlier version of the manuscript.


\bibliographystyle{alpha}
\bibliography{CSat}

\end{document}